\documentclass[a4paper,11pt]{article}
\usepackage[T1]{fontenc}
\usepackage{rotating}

\usepackage[title]{appendix}
\usepackage{slashbox}
\usepackage{upgreek}
\usepackage{amsmath} 
\usepackage{amsthm}
\usepackage{amsfonts}
\usepackage{mathtools}
\usepackage{amssymb}
\usepackage{enumitem}
\usepackage{gensymb}

\def\x{\mathbf{x}}
\def\y{\mathbf{y}}

\DeclareMathOperator{\Aut}{Aut}

\DeclareMathOperator{\sgn}{sgn}
\DeclareMathOperator{\Circ}{Circ}
\DeclareMathOperator{\Cay}{Cay}
\DeclareMathOperator{\Dih}{Dih}
\DeclareMathOperator{\twist}{\tau}
\DeclareMathOperator{\cart}{\Box}

\usepackage{graphicx}
\usepackage{url}
\usepackage[hyperfootnotes=true,colorlinks=true]{hyperref}
\hypersetup{
    colorlinks=true,
    linkcolor=blue,
    citecolor=magenta,      
    urlcolor=black
}

\usepackage{tikz}
\usetikzlibrary{calc}
\usetikzlibrary{backgrounds}
\usetikzlibrary{intersections}

\usepackage{authblk}
\usepackage{subcaption}
 
\usepackage[text={170mm,254mm}]{geometry}

\theoremstyle{plain}
\newtheorem{theorem}{Theorem}
\newtheorem{lemma}[theorem]{Lemma}
\newtheorem{corollary}[theorem]{Corollary}
\newtheorem{proposition}[theorem]{Proposition}
\newtheorem{conjecture}[theorem]{Conjecture}

\theoremstyle{definition}
\newtheorem{definition}{Definition}

\newtheorem{question}[theorem]{Question}

\newtheorem{example}{Example}

\usepackage{amsfonts}
\usepackage{amsmath}
\usepackage{ifxetex,ifluatex}
\if\ifxetex T\else\ifluatex T\else F\fi\fi T%
  \usepackage{fontspec}
\else
  \usepackage[T1]{fontenc}
  \usepackage[utf8]{inputenc}
  \usepackage{lmodern}
\fi

\usepackage{hyperref}

\title{Vertex and edge orbits in nut graphs}

\author[1,2,3]{Nino Ba\v{s}i\'{c}}
\author[4]{Patrick~W.~Fowler}
\author[1,2,3,5]{Tomaž Pisanski}

\affil[1]{FAMNIT, University of Primorska, Koper, Slovenia}
\affil[2]{IAM, University of Primorska, Koper, Slovenia}
\affil[3]{Institute of Mathematics, Physics and Mechanics, Ljubljana, Slovenia}
\affil[4]{Department of Chemistry, University of Sheffield, Sheffield S3 7HF, UK}
\affil[5]{FMF, University of Ljubljana, Slovenia}

\begin{document}

\maketitle

\begin{abstract}
A \emph{nut graph} is a simple graph for which the adjacency matrix has a single zero eigenvalue such that all non-zero kernel eigenvectors have no zero entry. 
If the isolated vertex is excluded as trivial, nut graphs have seven or more vertices; they are connected, non-bipartite, and have no leaves.
It is shown that a nut graph $G$ always has at least one more edge orbit than it has vertex orbits: $o_e(G) \geq o_v(G) + 1$, with
the obvious corollary that edge transitive nut graphs do not exist. We give infinite familes of vertex-transitive nut graphs with two orbits of edges, and infinite families of nut graphs with two 
orbits of vertices and three of edges.
Several constructions for nut graphs from smaller starting graphs are known: double subdivision of a bridge, four-fold subdivision of an edge, 
a construction for extrusion of a vertex with preservation of the degree sequence. To these we add multiplier constructions that yield nut graphs from regular (not necessarily nut graph) parents.
In general, constructions can have different effects on automorphism group and counts of vertex and edge orbits, but in the
case where the automorphism group is `preserved', they can be used in a predictable way to control vertex and edge orbit numbers.

\vspace{\baselineskip}

\vspace{\baselineskip}
\noindent
\textbf{Math.\ Subj.\ Class.\ (2020):} 
05C50, 
05C25, 
05C75, 
05C92 
\end{abstract}

\section{Introduction}

The main goal of the present paper is to find limitations on the numbers of orbits of vertices and edges
of nut graphs under the action of the full automorphism group, and in particular to show that every nut 
graph has more than one orbit of edges.
To substantiate this claim, we require some standard definitions.
All graphs considered in this paper are simple and connected.
By $\delta(G)$, $d(G)$ and $\Delta(G)$ we denote the minimum, average and maximum degrees of a vertex in graph $G$ 
(see \cite[Section~1.2]{Diestel2018}).
The adjacency matrix of graph $G$ is $\mathbf{A}(G)$
and the dimension of the nullspace of $\mathbf{A}(G)$
is the \emph{nullity},  $\eta(G)$.
Let $\Phi(M; \lambda)$ denote the characteristic polynomial of square matrix $M$.
The characteristic polynomial of graph $G$, denoted $\Phi(G; \lambda)$, is the characteristic
polynomial of its adjacency matrix, i.e.\ $\Phi(G; \lambda) = \Phi(\mathbf{A}(G); \lambda) = \det(\mathbf{A}(G) - \lambda \mathbf{I})$.
The spectrum of graph $G$ will be denoted $\sigma(G)$.
For a graph $G$ of order $n$, we  take $V(G) = \{1, 2, \ldots, n\}$.
The neighbourhood of a vertex $v$ in graph $G$ is denoted $N_G(v)$;
where the graph $G$ is clear from the context then we can simply write $N(v)$.
For other standard definitions we refer the reader to one of 
the many comprehensive treatments of graph spectra and related concepts (e.g.~\cite{Cvetkovic1997,Cvetkovic2010,Chung1997,Cvetkovic1995,haemers}).

\emph{Nut graphs} \cite{ScirihaGutman-NutExt} are graphs that have a one-dimensional nullspace (i.e., $\eta(G) = 1$), 
where the non-trivial kernel eigenvector $\x = [x_1\ \ldots\ x_n]^\intercal \in \ker \mathbf{A}(G)$ is full (i.e., $|x_i| > 0$ for all $i = 1, \ldots, n$). 
Nut graphs are connected, non-bipartite and have no
leaves (i.e.\ $\delta(G) \geq 2$ for every nut graph~$G$)~\cite{ScirihaGutman-NutExt}. 
As the defining paper considered the isolated vertex 
to be  a trivial case~\cite{ScirihaGutman-NutExt}, the \emph{non-trivial}
nut graphs have seven or more vertices. Nut graphs of small order have been enumerated 
(see e.g.~\cite{hog,HoG2} and \cite{CoolFowlGoed-2017}).
If $G$ is a \emph{regular} nut graph, then $\delta(G) = d(G) = \Delta(G) \geq 3$.
Note that there are no nut graphs with $\Delta(G) = 2$, as it is known that cycles are not nut graphs.
The case of $\Delta(G) = 3$ is of interest in chemical applications of graph theory, as a \emph{chemical graph} 
is a connected  graph with maximum degree at most three.
Chemical aspects of nut graphs are treated in \cite{nuttybook}.
The nut graph is a special case of the \emph{core graph}: 
a core graph is a graph with  $\eta(G)\ge 1$ for which it is possible to construct 
a kernel eigenvector in which all vertices of $G$ carry a non-zero entry.
Hence, a nut graph is a core graph of nullity one. Again,
$K_1$ is presumably a trival core graph in the standard definition.
Notice that a core graph may be bipartite or not, whereas a nut graph is not bipartite.

Let $G$ and $H$ be simple graphs. The cartesian product of $G$ and $H$, denoted $G \cart H$, is the graph with the vertex set
$\{ (u, v) \mid u \in V(G) \text{ and } v \in V(H) \}$ and the edge set $\{ (u, v)(u', v') \mid (uu' \in E(G) \text{ and }v = v') \text{ or }
(u = u' \text{ and } vv' \in E(H))  \}$. For further details on graph products see \cite{klavzar2011book}.

An \emph{automorphism} $\alpha$ of a graph $G$ is a permutation $\alpha\colon V(G) \to V(G)$ of the vertices
of $G$ that maps edges to edges and non-edges to non-edges.
The set of all automorphisms of a graph $G$ forms a group,
the (\emph{full}) \emph{automorphism group} of $G$, denoted by $\Aut(G)$.
The image of a vertex $v \in V(G)$ under automorphism $\alpha$ will be denoted $v^\alpha$.
Let $u, v \in V(G)$. If there is an automorphism $\alpha \in \Aut(G)$, such that $u^\alpha = v$,
vertices $u$ and $v$ belong to the same vertex orbit. This relation partitions the vertex set $V(G)$ into $o_v(G)$ vertex orbits.
Let $\{u_1,u_2\}, \{v_1,v_2\} \in E(G)$. If there is an automorphism $\alpha \in \Aut(G)$, such that 
$\{u_1^\alpha, u_2^\alpha\} = \{v_1, v_2\}$, then edges $u_1u_2$ and $v_1v_2$ belong to the same edge orbit.
This relation partitions the edge set $E(G)$ into $o_e(G)$ edge orbits.
See Figure~\ref{fig:orbit_count} for examples.
If $o_v(G) = 1$ (i.e.\ all vertices belong to the same vertex orbit) then the graph $G$ is said to be \emph{vertex transitive}.
Likewise, if $o_e(G) = 1$, then the graph $G$ is said to be \emph{edge transitive}.
A well known class of vertex-transitive graphs is the \emph{circulant graphs} \cite[Section 1.5]{Godsil2001}.
By $\Circ(\mathbb{Z}_n, S)$, where $S \subseteq \mathbb{Z}_n$, we denote the graph on 
the vertex set $V(G) = \mathbb{Z}_n$, where vertices $u, v \in V(G)$ are adjacent if and only if $| u - v | \in S$.
Let $\mathcal{G}$ be a subgroup of $\Aut(G)$.
The stabiliser of a vertex $v$ in $G$, denoted $\mathcal{G}_v$, is the subgroup of $\mathcal{G}$ that contains all elements $\alpha$ such that $v^\alpha = v$.
For other standard definitions from algebraic graph theory, we refer the reader to 
textbooks, e.g.~\cite{Godsil2001,Dobson2022,Biggs}. 

\begin{example}
There are three non-isomorphic nut graphs on $7$ vertices. We denote them $S_1, S_2$ and $S_3$ and call them the
Sciriha graphs. For each Sciriha graph the numbers of vertex and edge orbits and the order of the
(full) automorphism group are given; see Figure~\ref{fig:orbit_count}. \hfill$\diamond$
\begin{figure}[!htb]
\centering
\subcaptionbox{\label{subfig:1a}$\Omega(S_1) = (4, 5, 4)$}
{ \includegraphics[scale=1]{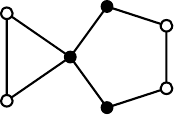} ~
\vspace{2em} }
 \qquad%
\subcaptionbox{\label{subfig:1b}$\Omega(S_2) = (4, 6, 4)$}
{ \quad \includegraphics[scale=1]{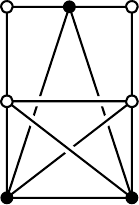} \quad}
 \qquad%
\subcaptionbox{\label{subfig:1c}$\Omega(S_3) = (3, 4, 6)$}
{ \includegraphics[scale=0.9]{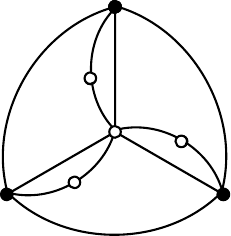} }
\caption{Vertex and edge orbits in the Sciriha graphs (i.e.\ the three nut graphs on $7$ vertices). For brevity, we denote
the triple $(o_v(G), o_e(G), |\mathrm{Aut}(G)|)$ by $\Omega(G)$. Vertices are coloured white or black to indicate equal and 
opposite entries in the kernel eigenvector.
Note that for the Sciriha graphs, the kernel eigenvector is totally symmetric in each case (i.e.\ its trace is $1$ for every automorphism).
}
\label{fig:orbit_count}
\end{figure}
\end{example}

\subsection{Natural questions about nut graphs}

Nut graphs are found within several well known graph classes, such as \emph{fullerenes}, \emph{cubic polyhedra}  \cite{feuerpaper} and
more general
\emph{regular graphs} \cite{Jan2020,basic2021regular}. 
Nut graphs within these classes tend to have low
symmetry, but attention has also been paid to finding nut graphs with high
symmetry (in the sense of having a small number of vertex orbits).
Recently, those pairs $(n, d)$ 
for which a $d$-regular nut circulant of order $n$ exists have
been characterised
in a series of papers~\cite{basic2021regular,DamnjanStevan2022,Damnjanovic2022a,Damnjanovic2022b,Damnjanovic2022c}. 
It is known that there are infinitely many vertex-transitive nut graphs~\cite{Jan2020}. 
It seems natural, therefore, to consider the possibility of
edge-transitive nut graphs.
(Recall that if $G$ is edge transitive, this does \emph{not} imply that $G$ is vertex transitive, \emph{nor} does it imply that $G$ is regular. 
However, an edge transitive connected graph has at most two vertex orbits.)

To give some background for our question, a preliminary  computer search based on the census of connected edge-transitive graphs on orders $n \leq 47$ \cite{Conder2019,marstonCensus} was conducted. It found no examples of nut graphs.
The census contains $1894$ graphs in total. Of these, $335$ graphs are non-singular and $2$ graphs have nullity $1$ (these graphs are $K_1$ and $P_3$).
In the census, there is at least one graph for every admissible nullity (i.e.\ for each
 $0 \leq k \leq 45$, there exists a graph $G$ with $\eta(G) = k$). There are $1312$ core graphs in the census (not counting $K_1$).
Amongst these, there are $1098$ bipartite graphs ($945$ non-regular, $25$ regular non-vertex-transitive, and
$128$ vertex-transitive graphs). The remaining $214$ non-bipartite edge-transitive core graphs are necessarily
vertex-transitive, but none of these are nut graphs. 
On this basis, it seems plausible to question whether edge-transitive nut graphs exist.
This prompted us to look for the general relationship between the numbers of vertex and edge orbits in nut graphs
that is proved in the next section.
In later sections, special attention is paid to nut graphs with one and two vertex orbits and the minimum number of edge orbits (respectively, two and three). Finally, the implications of constructions for the symmetry properties of nut graphs are investigated; a useful byproduct
is a simple proof that there exist infinitely many nut graphs for each even number of vertex orbits and any number of edge orbits
allowed by the main theorem.

\section{A relation between numbers of vertex and edge orbits}
\label{sec:generalResult}

The main result is embodied in the following theorem.

\begin{theorem}
\label{thm:norb}
Let $G$ be a nut graph. Then $o_e(G) \geq o_v(G) + 1$.
\end{theorem}

\noindent
This theorem immediately implies the following corollary.

\begin{corollary}
\label{thm:main}
Let $G$ be a nut graph. Then $G$ is not edge transitive.
\end{corollary}

It is, however, relatively
easy to find infinite families of vertex-transitive nut graphs with \emph{few} edge orbits.
For example, in Sections~{\color{red}\ref{subsec:fam12}} and~{\color{red}\ref{subsec:fam23}}, we provide infinite families of nut graphs for $(o_v, o_e)  = (1, 2)$ and $(o_v, o_e)  = (2, 3)$.
Moreover, as we saw from our examination of the
census, many core graphs are edge transitive. 

To prepare for the proof of Theorem~\ref{thm:main}, we recall some established results. 

\begin{lemma}[{\cite[Lemma~3.2.1]{Godsil2001}}]
Let $G$ be an edge-transitive graph with no isolated vertices. If $G$ is not vertex transitive, 
then $\Aut(G)$ has exactly two orbits, and these two orbits are a bipartition of $G$.
\end{lemma}

\par\noindent
A similar statement appears in \cite{Biggs} as Proposition~15.1.
A theorem from a previous investigation specifies necessary conditions relating 
order and degree of a vertex-transitive nut graph: 

\begin{theorem}[{\cite[Theorem 10]{Jan2020}}]
\label{thm:four}
Let $G$ be a vertex-transitive nut graph on $n$ vertices, of degree $d$.
Then $n$ and $d$ satisfy the following conditions. Either $d \equiv 0 \pmod{4}$,
and $n \equiv 0 \pmod{2}$ and $n \geq d + 4$; or $d \equiv 2 \pmod{4}$, and 
$n \equiv 0 \pmod{4}$ and $n \geq d + 6$.
\end{theorem}

\begin{lemma}
\label{lem:7}
Let $G$ be a vertex-transitive nut graph and let \,$\x = [x_1\ \ldots\ x_n]^\intercal \in \ker \mathbf{A}(G)$. 
Then the following statements hold:
\begin{enumerate}[label=(\alph*)]
\item\label{lem7_claima} $\x = \pm \x^\alpha$ for every $\alpha \in \Aut(G)$;
\item\label{lem7_claimb} $| x_i | = | x_j |$ for all $i$ and $j$;
\item\label{lem7_claimc} we can take the entries to be $x_i \in \{+1, -1\}$;
\item\label{lem7_claimd} $d(G)$ and $n$ are both even.
\end{enumerate}
\end{lemma}

\begin{proof}
As $G$ is vertex-transitive and $n \geq 7$, $G$ has a non-trivial automorphism group $\Aut(G)$, i.e.\ $|\Aut(G)| > 1$.
As $G$ is a nut graph, the kernel eigenvector $\x$ belongs to a one-dimensional eigenspace, and hence
spans a one-dimensional irreducible representation of $\Aut(G)$.
As the graph is vertex-transitive, each element $\alpha \in \Aut(G)$ sends
vertex $v$ to an image vertex $v^\alpha$ ($v^\alpha$ may be $v$).
The action of $\alpha$ on the vertices of $G$ can be extended to $\x$ by defining
$\x^\alpha = [x_{1^\alpha}\ \ldots\  x_{n^\alpha}]$.
As a one-dimensional irreducible representation
has trace either $+1$ or $-1$ under any particular
automorphism  $\alpha \in \Aut(G)$, it follows that $\x = \pm \x^\alpha$. This proves claim (a).
In particular, $|x_1| = |x_{1^\alpha}|$.
Since for every vertex $v$ there exists an $\alpha \in \Aut(G)$ such that $v = 1^\alpha$, the claim (b) follows.
To show that the entries in the kernel eigenvector can be drawn from the set $\{+1, -1\}$, it is enough to
normalise the vector to $x_1 = 1$, verifying claim (c).
to establish claim (d), note that
the local condition for entries of the vector $\x \in \ker \mathbf{A}(H)$ is
\begin{equation}
\label{eq:localCondition}
\sum_{u \in N(v)} x_u = 0 \quad \text{for }v = 1, \ldots, n.
\end{equation}
As all entries of $\x$ are in $\{ +1, -1 \}$, every vertex $v$ must be of \emph{even degree}.
Since $G$ is a regular graph, the entries of the Perron vector $\y$ of $G$ (i.e.\ the eigenvector that corresponds to the largest
eigenvalue $\lambda_1$) are all equal to $+1$.
As $\x$ is orthogonal to $\y$, i.e.\ $\sum_{i=1}^{n} x_i = 0$,
there are equal numbers of $+1$ and $-1$ entries in $\x$, and $G$ must have even order $n$, completing claim (d).
\end{proof}

We note that the arguments used in claims \ref{lem7_claima} to \ref{lem7_claimc} in the proof of Lemma~\ref{lem:7} can be applied 
orbit-wise for graphs that are not vertex transitive. Lemma~\ref{lem:7} thus generalises naturally to the following:

\begin{lemma}
\label{lem:7gen}
Let $G$ be a nut graph and let \,$\x = [x_1\ \ldots\ x_n]^\intercal \in \ker \mathbf{A}(G)$. 
Then the following statements hold:
\begin{enumerate}[label=(\alph*)]
\item $\x = \pm \x^\alpha$ for every $\alpha \in \Aut(G)$;
\item $| x_i | = | x_j |$ if $i$ and $j$ belong to the same vertex orbit;
\item we can take the entries to be $x_i \in \{+a_j, -a_j\}$ if $i \in \mathcal{V}_j$, where $a_j$ is a non-zero constant for orbit~$\mathcal{V}_j$.
\end{enumerate}
\end{lemma}

This lemma will be used in the proof of Theorem~\ref{thm:norb}.
Before the proof, we introduce some definitions. The first of these deal with edges.
Let $G$ be a nut graph with $k$ vertex orbits $V(G) = \mathcal{V}_1 \sqcup \mathcal{V}_2 \sqcup \cdots \sqcup \mathcal{V}_k$. 
There are several \emph{types} of edge
(as indicated schematically in Figure~\ref{fig:edge_types2}
for the case $k = 3$):
\begin{enumerate}[label=(\alph*)]
\item \emph{intra-orbit} edge types, where both endvertices of an edge are in the same orbit $\mathcal{V}_i$, denoted $e_{i}$;
\item \emph{inter-orbit} edge types, where endvertices of an edge are in two different orbits $\mathcal{V}_i$ and $\mathcal{V}_j$, $i \neq j$, denoted $e_{ij} = e_{ji}$.
\end{enumerate}
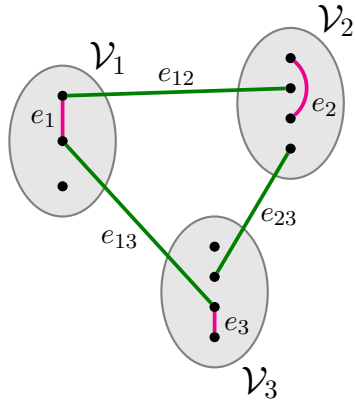
\begin{figure}[!htb]
\centering
\begin{tikzpicture}
\tikzstyle{edge}=[draw,line width=1.4pt]
\tikzstyle{every node}=[draw, circle, fill=black, inner sep=1.2pt]
\draw[thick,color=gray,fill=gray!20!white] (0,0) ellipse (0.7cm and 1cm);
\draw[thick,color=gray,fill=gray!20!white] (3,0.5) ellipse (0.7cm and 1cm);
\draw[thick,color=gray,fill=gray!20!white] (2,-2) ellipse (0.7cm and 1cm);
\node[draw=none,fill=none] at (0.5+0.1, 1.1) {\Large $\mathcal{V}_1$};
\node[draw=none,fill=none]  at (3.5+0.1, 1.6) {\Large $\mathcal{V}_2$};
\node[draw=none,fill=none]  at (2.5+0.1, -3.2) {\Large $\mathcal{V}_3$};
\node (v1) at (0,0.6) {}; 
\node (v2) at (0,0.0) {}; 
\node (v3) at (0,-0.6) {}; 
\node (v4) at (3,0.6+0.5) {}; 
\node (v5) at (3,0.2+0.5) {}; 
\node (v6) at (3,-0.2+0.5) {}; 
\node (v7) at (3,-0.6+0.5) {}; 
\node (v8) at (2,0.6-2.0) {}; 
\node (v9) at (2,0.2-2.0) {}; 
\node (v10) at (2,-0.2-2.0) {}; 
\node (v11) at (2,-0.6-2.0) {}; 
\draw[edge,color=green!50!black] (v2) -- (v10);
\draw[edge,color=green!50!black] (v1) -- (v5);
\draw[edge,color=green!50!black] (v9) -- (v7);
\draw[edge,color=magenta] (v1) -- (v2);
\draw[edge,color=magenta] (v10) -- (v11);
\draw[edge,color=magenta] (v4) edge[bend left=50] (v6);
\node[draw=none,fill=none] at (1.5, 0.85) {$e_{12}$};
\node[draw=none,fill=none] at (0.75, -1.3) {$e_{13}$};
\node[draw=none,fill=none] at (2.85, -1.0) {$e_{23}$};
\node[draw=none,fill=none] at (2.3, -2.45) {$e_{3}$};
\node[draw=none,fill=none] at (3.45, 0.45) {$e_{2}$};
\node[draw=none,fill=none] at (-0.25, 0.3) {$e_{1}$};
\end{tikzpicture}
\caption{Schematic representation of a graph with three vertex orbits (represented by the three bags of vertices). There are six edge types, denoted 
$e_{1}, e_{2}, e_{3}$ (intra-orbit), $e_{12}, e_{13}$ and $e_{23}$ (inter-orbit).}
\label{fig:edge_types2}
\end{figure}
It is clear that edges of different types cannot belong to the same edge orbit. Moreover, a single edge type may comprise several edge orbits.

An additional definition will also be useful for the proof.
Let the \emph{vertex-orbit graph} of $G$, denoted $\mathfrak{G}(G)$, be the graph whose vertices are vertex orbits of $G$ and
vertices $\mathcal{V}_i$ and 
$\mathcal{V}_j$ are adjacent in $\mathfrak{G}(G)$ if there exists at least one edge of type $e_{ij}$ in $G$. 
Note that in our case $\mathfrak{G}(G)$ contains $k$ vertices. Moreover, the graph $\mathfrak{G}(G)$ is a simple graph, not to
be confused with the orbit graph as defined in~\cite{PisanskiBook}, which is typically a pregraph.

\begin{lemma}
\label{lem:allzerolemma}
Let $G$ be a nut graph with $k$ vertex orbits $V(G) = \mathcal{V}_1 \sqcup \mathcal{V}_2 \sqcup \cdots \sqcup \mathcal{V}_k$
and let \,$\x = [x_1\ \ldots\ x_n]^\intercal \in \ker \mathbf{A}(G)$. 
Suppose there exists a vertex orbit $\mathcal{V}_\ell$ in $G$, such that 
\begin{enumerate}[label=(\roman*)]
\item $\mathcal{V}_\ell$ is a leaf in $\mathfrak{G}(G)$, and
\item vertices $\mathcal{V}_\ell$ form an independent set in $G$.
\end{enumerate}
Then for every orbit $\mathcal{V}_i$ of $G$ it holds that
\begin{equation}
\label{eq:allzero}
\sum_{j \in \mathcal{V}_i} x_j = 0.
\end{equation}
Moreover, for each orbit $\mathcal{V}_i$ it holds that $|\{ j \in \mathcal{V}_i : x_j > 0 \}| = |\{ j \in \mathcal{V}_i : x_j < 0 \}|$ and
the size of the orbit $\mathcal{V}_i$ is even.
\end{lemma}

\begin{proof}
Let $\mathcal{V}_{\ell'}$ be the neighbour of $\mathcal{V}_{\ell}$ in $\mathfrak{G}(G)$. Let $d_{ij}$ be the number of neighbours of a vertex $v \in \mathcal{V}_i$ that reside in $\mathcal{V}_j$. The local condition says that
\begin{equation}
\sum_{u \in N(v)} x_u = 0 \quad \text{for }v \in \mathcal{V}_{\ell}.
\end{equation}
Therefore,
\begin{equation}
\sum_{v \in \mathcal{V}_\ell} \sum_{u \in N(v)} x_u = \sum_{u \in \mathcal{V}_{\ell'}} d_{\ell'\ell} x_u = 0.
\label{eq:leafy}
\end{equation}
This implies that $$\sum_{u \in \mathcal{V}_{\ell'}} x_u = 0.$$
Hence, the orbit $\mathcal{V}_\ell'$ must contain at least one vertex $v$ with $x_v > 0$ and at least one vertex $w$ with $x_w < 0$.
This implies that there exists a sign-reversing $\alpha \in \Aut(G)$. Within each orbit, $\alpha$ maps vertices with positive entries
in the kernel eigenvector to vertices with negative entries, and vice-versa. Therefore, the cardinalities of these two sets of vertices are equal.
Equation \eqref{eq:allzero} follows by Lemma~\ref{lem:7gen}(b), and the claim about the parity is evident.
\end{proof}

We can now proceed to the proof of the theorem.
\begin{proof}[Proof of Theorem~\ref{thm:norb}]
Let $G$ be a nut graph with $k$ vertex orbits $V(G) = \mathcal{V}_1 \sqcup \mathcal{V}_2 \sqcup \cdots \sqcup \mathcal{V}_k$. 

If the graph $G$ is connected then $\mathfrak{G}(G)$ is also connected.
Connectedness of $\mathfrak{G}(G)$ implies that $o_e(G) \geq o_v(G)-1$ \cite{Buset1985}, since a connected graph on 
$k$ vertices has at least $k - 1$ edges and each
edge of $\mathfrak{G}(G)$ gives rise to at least one edge orbit of $G$.
Suppose there are no intra-orbit edges in $G$. In this case $G$ is bipartite if and only if $\mathfrak{G}(G)$ is bipartite. 
But a bipartite graph is not a nut graph.
Hence we have $o_e(G) \ge o_v(G)$. 
 To avoid bipartiteness we can do one of two things: 
\begin{enumerate}[label=(\Roman*)]
\item
We may add intra-orbit edges to one or more vertex orbits. We need only consider addition of one such edge type,
as addition of two or more would already imply $o_e(G) \geq o_v(G) + 1$.
\item
 We may add another type of inter-orbit edge to make an odd cycle in $\mathfrak{G}(G)$. Note that $\mathfrak{G}(G)$ becomes a unicyclic graph.
Again, we do not need to consider addition of more than one edge type.
\end{enumerate}

Suppose that $\mathfrak{G}(G)$ contains a leaf $\mathcal{V}_\ell$ that is an independent set in $G$ (in other words,
there are no intra-orbit edges in $\mathcal{V}_\ell$). Let $\mathcal{V}_\ell'$ be the neighbour of
$\mathcal{V}_\ell$ in $\mathfrak{G}(G)$.
By Lemma~\ref{lem:allzerolemma}, the numbers of positive and negative entries in the kernel eigenvector are equal within any given orbit.
This implies the existence of a sign-reversing automorphism $\alpha \in \Aut(G)$.
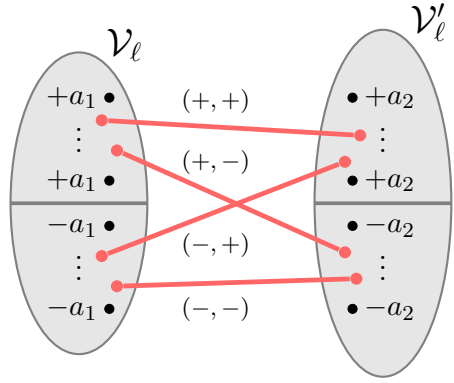
\begin{figure}[!htb]
\centering
\begin{tikzpicture}
\tikzstyle{edge}=[draw,line width=1.4pt]
\tikzstyle{every node}=[draw, circle, fill=black, inner sep=1.2pt]
\draw[thick,color=gray,fill=gray!20!white] (0,0) ellipse (0.9cm and 2cm);
\draw[thick,color=gray,fill=gray!20!white] (4,0.0) ellipse (0.9cm and 2.3cm);
\draw[draw,line width=1.4pt,color=gray] (-0.9, 0) -- (0.9, 0);
\draw[draw,line width=1.4pt,color=gray] (4-0.9, 0) -- (4+0.9, 0);
\node[draw=none,fill=none] at (0.5+0.1, 2.1) {\Large $\mathcal{V}_\ell$};
\node[draw=none,fill=none]  at (4.5+0.1, 2.4) {\Large $\mathcal{V}_\ell'$};
\node[label=0:{$+a_2$}] (v1) at (3.6,1.4) {}; 
\node[label=0:{$+a_2$}] (v2) at (3.6,0.3) {}; 
\node[label=0:{$-a_2$}] (v3) at (3.6,-1.4) {}; 
\node[label=0:{$-a_2$}] (v4) at (3.6,-0.3) {}; 
\node[draw=none,fill=none] at (4, 0.95) {$\vdots$};
\node[draw=none,fill=none] at (4, -0.75) {$\vdots$};
\node[label=180:{$+a_1$}] (v5) at (0.4,1.4) {}; 
\node[label=180:{$+a_1$}] (v6) at (0.4,0.3) {}; 
\node[label=180:{$-a_1$}] (v7) at (0.4,-1.4) {}; 
\node[label=180:{$-a_1$}] (v8) at (0.4,-0.3) {}; 
\node[draw=none,fill=none] at (0, 0.95) {$\vdots$};
\node[draw=none,fill=none] at (0, -0.75) {$\vdots$};
\node[fill=red!60!white,draw=none,inner sep=1.6pt] (u1) at (3.7,0.9) {}; 
\node[fill=red!60!white,draw=none,inner sep=1.6pt] (u2) at (0.3,1.1) {}; 
\draw[edge,color=red!60!white,line width=1.8pt] (u1) -- (u2);
\node[draw=none,fill=none] at (1.8, 1.35) {\footnotesize $(+,+)$};
\node[fill=red!60!white,draw=none,inner sep=1.6pt] (u3) at (3.5,0.55) {}; 
\node[fill=red!60!white,draw=none,inner sep=1.6pt] (u4) at (0.3,-0.7) {}; 
\draw[edge,color=red!60!white,line width=1.8pt] (u3) -- (u4);
\node[draw=none,fill=none] at (1.8, -0.55) {\footnotesize $(-,+)$};
\node[fill=red!60!white,draw=none,inner sep=1.6pt] (u5) at (3.5,-0.65) {}; 
\node[fill=red!60!white,draw=none,inner sep=1.6pt] (u6) at (0.5,0.7) {}; 
\draw[edge,color=red!60!white,line width=1.8pt] (u5) -- (u6);
\node[draw=none,fill=none] at (1.8, 0.55) {\footnotesize $(+,-)$};
\node[fill=red!60!white,draw=none,inner sep=1.6pt] (u7) at (3.65,-1) {}; 
\node[fill=red!60!white,draw=none,inner sep=1.6pt] (u8) at (0.5,-1.1) {}; 
\draw[edge,color=red!60!white,line width=1.8pt] (u7) -- (u8);
\node[draw=none,fill=none] at (1.8, -1.4) {\footnotesize $(-,-)$};
\end{tikzpicture}
\caption{Vertex orbits $\mathcal{V}_\ell$ and $\mathcal{V}_\ell'$ as defined in the proof of Theorem~\ref{thm:norb},
showing the four possible signatures for edges of type $e_{\ell\ell'}$.}
\label{fig:edge_signatures_2}
\end{figure}
Each edge of type $e_{\ell\ell'}$ can be assigned one of four signatures according to signs of the kernel eigenvector entries for its endvertices (shown schematically in Figure~\ref{fig:edge_signatures_2}). We now consider the action of the automorphism $\alpha$ on edges of each
signature. 
Since $\alpha$ is sign-reversing, edge signatures are swapped: $(+,+) \leftrightarrow (-, -)$ and $(+,-) \leftrightarrow (-, +)$. Hence, edges of
type $e_{\ell\ell'}$ fall into at least two orbits, determined by relative sign of endvertex entries:
a \emph{like} edge is of signature $(+,+)$ or $(-,-)$, while an \emph{unlike} edge
is of signature $(+, -)$ or $(-, +)$.
 By the local condition at a vertex of $\mathcal{V}_\ell$,
the presence of a $(+,+)$ edge implies the presence of a $(+,-)$ edge and vice-versa, hence the two corresponding edge orbits are both non-empty. As there is at least one edge orbit included within $e_{\ell\ell'}$, the number of edge orbits  in $G$ is greater than
the number of its edge types.

This proves case (I) and also case (II) where $\mathfrak{G}(G)$ is a unicyclic graph but not a cycle. If $\mathfrak{G}(G)$ is a cycle (necessarily odd) 
and there are no inter-orbit edges, a different
argument is needed. 
Recall that no automorphism maps a like to an unlike edge (or vice versa), so they cannot be
in the same edge orbit. If $e_{i,i+1}$ contains both like and unlike edges, this immediately implies $o_e(G) \geq o_v(G) + 1$. So, for every $i$, we can assume that $e_{i,i+1}$ contains only like or only unlike edges. Take any vertex $u \in \mathcal{V}_i$. Since there are no intra-edges it has
to be connected to neighbours in $\mathcal{V}_{i - 1}$ via like and neighbours in $\mathcal{V}_{i + 1}$ via unlike edges or vice versa.
Therefore, the edges of $\mathfrak{G}(G)$ can be properly coloured with colours `like' and `unlike'. 
But $\mathfrak{G}(G)$ is an odd cycle, so no such edge colouring
exists. Hence, at least one type $e_{i,i+1}$ contains edges of both kinds.
\end{proof}

Lemma~\ref{lem:allzerolemma} implies that a sign-reversing automorphism exists in a nut graph if at least one vertex orbit $\mathcal{V}_\ell$ is a leaf in $\mathfrak{G}(G)$
 and $\mathcal{V}_\ell$ has no intra-orbit edges. A similar structural result can also be obtained if $\mathfrak{G}(G)$ is an odd cycle and $G$ has no intra-orbit edges. 

\begin{proposition}
\label{prop:8}
Let $G$ be a nut graph with $k$ vertex orbits $V(G) = \mathcal{V}_1 \sqcup \mathcal{V}_2 \sqcup \cdots \sqcup \mathcal{V}_k$
and let \,$\x = [x_1\ \ldots\ x_n]^\intercal \in \ker \mathbf{A}(G)$. 
Suppose that every $\mathcal{V}_\ell$ forms an independent set in $G$ and $\mathfrak{G}(G)$ is an odd cycle. Then for every orbit $\mathcal{V}_i$ of $G$ it holds that
\begin{equation}
\label{eq:allzero}
\sum_{j \in \mathcal{V}_i} x_j = 0.
\end{equation}
Moreover, for each orbit $\mathcal{V}_i$ it holds that $|\{ j \in \mathcal{V}_i : x_j > 0 \}| = |\{ j \in \mathcal{V}_i : x_j < 0 \}|$ and
the size of the orbit $\mathcal{V}_i$ is even.
\end{proposition}

\begin{lemma}
\label{lem:detLem}
Let $n$ be an odd integer and let $\mathbf{A} = [a_{i,j}]_{1 \leq i,j \leq n}$ be a $n \times n$ matrix such that $a_{i,j} = 0$ unless
$(i, j) \in \{(1, n), (n, 1)\}$ or $| i - j | = 1$.
Then $\det \mathbf{A} = a_{2,1}a_{3,2} \cdots a_{n,n-1}a_{1,n} + a_{1,2}a_{2,3} \cdots a_{n-1,n}a_{n,1}$.
\end{lemma}

\begin{proof}
Recall that by definition
\begin{equation}
\det \mathbf{A} = \sum_{\sigma \in S_n} \left( \sgn(\sigma) \prod_{i=1}^{n} a_{i,\sigma(i)} \right),
\end{equation}
where $S_n$ is the set of all permutations of length $n$. Note that the product $\prod_{i=1}^{n} a_{i,\sigma(i)}$ necessarily
contains a zero factor, unless $\sigma \in \{(1\ 2\ 3\ \ldots\ n), (1\ n\ n-1\ \ldots\ 2)\}$.
\end{proof}

\begin{proof}[Proof of Proposition~\ref{prop:8}]
If necessary, relabel the orbits $\mathcal{V}_1, \mathcal{V}_2, \ldots, \mathcal{V}_k$, so that 
$\mathcal{V}_i$ and $\mathcal{V}_{i + 1}$ are neighbours in the cycle $\mathfrak{G}(G)$.
Let $d_{ij}$ be the number of neighbours of a vertex $v \in \mathcal{V}_i$ that reside in $\mathcal{V}_j$. The local condition 
gives us one equation for each orbit, namely
\begin{equation}
 \sum_{v \in \mathcal{V}_i}\sum_{u \in N(v)} x_u = \sum_{u \in \mathcal{V}_{i-1}}  d_{i-1,i} x_u + \sum_{u \in \mathcal{V}_{i+1}} d_{i+1,i} x_u = 0,
\qquad(1 \leq i \leq k)
\end{equation}
where we consider indices modulo $k$. Let us define $s_i = \sum_{u \in \mathcal{V}_{i}} x_u$ for $i = 1, \ldots, k$. We have 
the matrix equation
\begin{equation}
\label{eq:matrixEq}
\begin{bmatrix}
0 & d_{3,2} & 0 & \ldots & 0 & d_{k,1}\\
d_{1,2} & 0 & d_{4,3} & \ddots & \vdots & 0 \\
0 & d_{2,3} & 0 & d_{5,4} & 0 & \vdots \\
\vdots & 0 & d_{3,4} & 0 & \ddots & 0 \\
0 & \vdots & \ddots & \ddots & \ddots & d_{1,k} \\
d_{2, 1} & 0 & \ldots & 0 & d_{k-1,k} & 0 
\end{bmatrix}
\begin{bmatrix}
s_1 \\
s_2 \\
\vdots \\
\vdots \\
\vdots \\
s_k
\end{bmatrix} = \mathbf{0}_{k\times 1}.
\end{equation}
By Lemma~\ref{lem:detLem}, the determinant of the square matrix in Equation~\eqref{eq:matrixEq} is 
$$d_{2,1}d_{3,2} \cdots d_{k,k-1}d_{1,k} + d_{1,2}d_{2,3} \cdots d_{k-1,k}d_{k,1} > 0,$$ 
since $d_{2,1}, d_{3,2}, \ldots$ are all positive. Hence, $s_1 = s_2 = \cdots = s_k = 0$. This already implies
the existence of a sign-reversing automorphism and the fact that 
$|\{ j \in \mathcal{V}_i : x_j > 0 \}| = |\{ j \in \mathcal{V}_i : x_j < 0 \}|$.
\end{proof}

\section{Vertex transitive nut graphs}

We have seen that $o_e(G) = o_v(G) = 1$ is not possible for a nut graph $G$. 
However,  many other possibilities for $o_e(G)$ may exist. 
First, we filtered out all nut graphs from databases of small vertex-transitive graphs on up to $n \leq 46$ vertices \cite{RH2020,HOLT2019}.
The counts are shown in Table~\ref{table:VTnutstats}. 
Recall that a vertex-transitive graph $G$ is a nut graph if and only if $\eta(G) = 1$, so this search requires only
computation of the nullity and moreover, by Theorem~\ref{thm:four}, can be limited to graphs of even order and degree.
As the table shows, most of these vertex transitive graphs are connected
and a significant proportion of vertex transitive graphs of even order are nut graphs.
As a preliminary survey of symmetry aspects, 
we calculated the number of edge orbits for all vertex-transitive nut graphs; see Table~\ref{tbl:1}, which has a number of interesting features.
It has only zero entries for $o_e(G)=1$, as demanded by Theorem~\ref{thm:main}, but there is no apparent restriction
on the values of $o_e(G)$ that can occur for large enough order of $G$. Note that a vertex-transitive nut graph $G$ with a large $o_e(G)$ must have large degree. 
To place these results in context, we also calculated the number of edge orbits of connected vertex-transitive graphs of even order. See Table~\ref{tbl:1full}.
We see some 
intriguing gaps in Table~\ref{table:VTnutstats} for particular pairs $(n, o_e)$, e.g.\ $(n, o_e) \in \{(22, 3), (22, 5), (22, 7), (22, 10), (22, 11)\}$, even
though 
the numbers of vertex-transitive graphs for these pairs of parameters are $37, 115, 138, 50$ and
$23$, respectively. 

\begin{table}
\centering
\begin{tabular}{|r|r|r|r|r|}
\hline
$n$	& All VT	& Connected VT	& VT nut graphs & Proportion \\
\hline
\hline
8	& 14	& 10	& 1	& 10.00\% \\
10	& 22	& 18	& 1	& 5.56\% \\
12	& 74	& 64	& 4	& 6.25\% \\
14	& 56	& 51	& 5	& 9.80\% \\
16	& 286	& 272	& 20	& 7.35\% \\
18	& 380	& 365	& 23	& 6.30\% \\
20	& 1214	& 1190	& 150	& 12.61\% \\
22	& 816	& 807	& 101	& 12.52\% \\
24	& 15506	& 15422	& 1121	& 7.27\% \\
26	& 4236	& 4221	& 508	& 12.04\% \\
28	& 25850	& 25792	& 4793	& 18.58\% \\
30	& 46308	& 46236	& 3146	& 6.80\% \\
32	& 677402	& 677116	& 47770	& 7.05\% \\
34	& 132580	& 132543	& 14565	& 10.99\% \\
36	& 1963202	& 1962756	& 214391	& 10.92\% \\
38	& 814216	& 814155	& 85234	& 10.47\% \\
40	&13104170	& 13102946	& 1815064	& 13.85\% \\
42	& 9462226	& 9461929	& 693416	& 7.33\% \\
44	& 39134640	& 39133822	& 7376081	& 18.85\% \\
46	& 34333800	& 34333611	& 3281206	& 9.56\% \\
\hline
\end{tabular}
\caption{The number of nut graphs among vertex transitive (VT) graphs on even orders $8 \leq n \leq 46$.
The final column is the ratio between the number of VT nut graphs and the number of connected VT graphs on
a given order, expressed as a percentage.
}
\label{table:VTnutstats}
\end{table}

\begin{sidewaystable}[!p] 
\centering
\footnotesize
\begin{tabular}{|c||rrrrrrrrrrrrrrrrrrrr|}
\hline 
\backslashbox{$o_e$}{$n$}  & 8 & 10 & 12 & 14 & 16 & 18 & 20 & 22 & 24 & 26 & 28 & 30 & 32 & 34 & 36 & 38 & 40 & 42 & 44 & 46 \\
\hline 
\hline 
1 & 0 & 0 & 0 & 0 & 0 & 0 & 0 & 0 & 0 & 0 & 0 & 0 & 0 & 0 & 0 & 0 & 0 & 0 & 0 & 0 \\
2 & 1 & 1 & 2 & 2 & 3 & 4 & 6 & 4 & 6 & 8 & 8 & 7 & 10 & 11 & 15 & 10 & 16 & 14 & 12 & 10 \\
3 & 0 & $0$ & 0 & 0 & 3 & 1 & 3 & 0 & 12 & 7 & 7 & 11 & 39 & 14 & 46 & 5 & 58 & 39 & 21 & 0 \\
4 & $-$ & $0$ & 0 & 2 & 1 & 7 & 9 & 18 & 51 & 36 & 31 & 73 & 118 & 100 & 209 & 142 & 407 & 280 & 186 & 270 \\
5 & $-$ & $-$ & 2 & 0 & 5 & 1 & 20 & 0 & 93 & 0 & 79 & 47 & 332 & 0 & 376 & 0 & 1079 & 341 & 349 & 0 \\
6 & $-$ & $-$ & $-$ & 1 & 6 & 6 & 32 & 38 & 164 & 119 & 277 & 258 & 1175 & 632 & 1604 & 1116 & 4349 & 2244 & 3944 & 3285 \\
7 & $-$ & $-$ & $-$ & $-$ & 2 & 0 & 30 & 0 & 181 & 0 & 306 & 98 & 1457 & 0 & 2414 & 0 & 6854 & 1747 & 4512 & 0 \\
8 & $-$ & $-$ & $-$ & $-$ & $0$ & 0 & 21 & 4 & 131 & 34 & 312 & 171 & 2250 & 600 & 4181 & 1750 & 14674 & 6410 & 12993 & 9870 \\
9 & $-$ & $-$ & $-$ & $-$ & $0$ & 4 & 16 & 32 & 222 & 186 & 756 & 1078 & 4788 & 2363 & 10659 & 6270 & 37144 & 24419 & 44984 & 31680 \\
10 & $-$ & $-$ & $-$ & $-$ & $-$ & $0$ & 9 & 0 & 97 & 5 & 505 & 70 & 5205 & 385 & 10743 & 1750 & 54653 & 11690 & 55182 & 19362 \\
11 & $-$ & $-$ & $-$ & $-$ & $-$ & $-$ & 3 & 0 & 100 & 0 & 924 & 23 & 8242 & 0 & 26197 & 0 & 110092 & 6215 & 198328 & 0 \\
12 & $-$ & $-$ & $-$ & $-$ & $-$ & $-$ & 1 & 5 & 41 & 105 & 755 & 1013 & 8438 & 6042 & 33238 & 27737 & 171053 & 130227 & 373849 & 334181 \\
13 & $-$ & $-$ & $-$ & $-$ & $-$ & $-$ & $-$ & $0$ & 20 & 0 & 476 & 1 & 6536 & 0 & 32259 & 0 & 209056 & 4405 & 509682 & 0 \\
14 & $-$ & $-$ & $-$ & $-$ & $-$ & $-$ & $-$ & $-$ & 3 & 0 & 197 & 0 & 3716 & 7 & 25059 & 140 & 179657 & 5573 & 507008 & 9870 \\
15 & $-$ & $-$ & $-$ & $-$ & $-$ & $-$ & $-$ & $-$ & $0$ & 8 & 110 & 284 & 3249 & 3807 & 28390 & 33151 & 298457 & 258510 & 1071473 & 1034877 \\
16 & $-$ & $-$ & $-$ & $-$ & $-$ & $-$ & $-$ & $-$ & $-$ & $0$ & 39 & 0 & 1238 & 0 & 14434 & 8 & 184646 & 1503  & 599584 & 2460 \\
17 & $-$ & $-$ & $-$ & $-$ & $-$ & $-$ & $-$ & $-$ & $-$ & $-$ & 9 & 0 & 682 & 0 & 13486 & 0 & 213377 & 739  & 1213286 & 0 \\
18 & $-$ & $-$ & $-$ & $-$ & $-$ & $-$ & $-$ & $-$ & $-$ & $-$ & 2 & 12 & 241 & 588 & 6936 & 11982 & 150605 & 198368 & 987503  & 1225073 \\
19 & $-$ & $-$ & $-$ & $-$ & $-$ & $-$ & $-$ & $-$ & $-$ & $-$ & $-$ & $0$ & 52 & 0 & 2790 & 0 & 93113 & 101& 712011 & 0  \\
20 & $-$ & $-$ & $-$ & $-$ & $-$ & $-$ & $-$ & $-$ & $-$ & $-$ & $-$ & $-$ & 2 & 0 & 883 & 0 & 44117 & 3 & 451854 & 10 \\
21 & $-$ & $-$ & $-$ & $-$ & $-$ & $-$ & $-$ & $-$ & $-$ & $-$ & $-$ & $-$ & $0$ & 16 & 379 & 1152 & 27428 & 38386 & 345284 & 528000 \\
22 & $-$ & $-$ & $-$ & $-$ & $-$ & $-$ & $-$ & $-$ & $-$ & $-$ & $-$ & $-$ & $0$ & $0$ & 70 & 0 & 9106 & 0 & 125961 & 0  \\
23 & $-$ & $-$ & $-$ & $-$ & $-$ & $-$ & $-$ & $-$ & $-$ & $-$ & $-$ & $-$ & $-$ & $-$ & 21 & 0 & 4024 & 0 & 102696 & 0 \\
24 & $-$ & $-$ & $-$ & $-$ & $-$ & $-$ & $-$ & $-$ & $-$ & $-$ & $-$ & $-$ & $-$ & $-$ & 2 & 21 & 845 & 2171 & 37414  & 78705 \\
25 & $-$ & $-$ & $-$ & $-$ & $-$ & $-$ & $-$ & $-$ & $-$ & $-$ & $-$ & $-$ & $-$ & $-$ & $-$ & $0$ & 225 & 0 & 12892 & 0  \\
26 & $-$ & $-$ & $-$ & $-$ & $-$ & $-$ & $-$ & $-$ & $-$ & $-$ & $-$ & $-$ & $-$ & $-$ & $-$ & $-$ & 28 & 0 & 3862 & 0  \\
27 & $-$ & $-$ & $-$ & $-$ & $-$ & $-$ & $-$ & $-$ & $-$ & $-$ & $-$ & $-$ & $-$ & $-$ & $-$ & $-$ & 1 & 31  & 972 & 3520  \\
28 & $-$ & $-$ & $-$ & $-$ & $-$ & $-$ & $-$ & $-$ & $-$ & $-$ & $-$ & $-$ & $-$ & $-$ & $-$ & $-$ & $-$ & $0$  & 206  & 0  \\
29 & $-$ & $-$ & $-$ & $-$ & $-$ & $-$ & $-$ & $-$ & $-$ & $-$ & $-$ & $-$ & $-$ & $-$ & $-$ & $-$ & $-$ & $-$  & 29 & 0  \\
30 & $-$ & $-$ & $-$ & $-$ & $-$ & $-$ & $-$ & $-$ & $-$ & $-$ & $-$ & $-$ & $-$ & $-$ & $-$ & $-$ & $-$ & $-$  & 4 & 33  \\
31 & $-$ & $-$ & $-$ & $-$ & $-$ & $-$ & $-$ & $-$ & $-$ & $-$ & $-$ & $-$ & $-$ & $-$ & $-$ & $-$ & $-$ & $-$ & $-$ & $0$ \\
\hline 
\hline 
$\Sigma$ & $1$ & $1$ & $4$ & $5$ & $20$ & $23$ & $150$ & $101$ & $1121$ & $508$ & $4793$ & $3146$ & $47770$ & $14565$ & $214391$ & $85234$ & $1815064$ & $693416$ & $7376081$ & $3281206$ \\
\hline 
\end{tabular}
\caption{The number of vertex-transitive nut graphs of the given order $n$ and number of edge orbits $o_e$.}
\label{tbl:1}
\end{sidewaystable}

\begin{sidewaystable}[!p] 
\centering
\footnotesize
\begin{tabular}{|c||rrrrrrrrrrrrrrrrrrrr|}
\hline 
\backslashbox{$o_e$}{$n$}  & 8 & 10 & 12 & 14 & 16 & 18 & 20 & 22 & 24 & 26 & 28 & 30 & 32 & 34 & 36 & 38 & 40 & 42 & 44 & 46 \\
\hline 
\hline 
1 & 5 & 8 & 11 & 8 & 15 & 14 & 22 & 8 & 34 & 13 & 26 & 41 & 42 & 10 & 69 & 10 & 71 & 56 & 16 & 7 \\
2 & 4 & 6 & 24 & 11 & 34 & 53 & 79 & 15 & 249 & 38 & 101 & 263 & 334 & 42 & 585 & 37 & 645 & 398 & 104 & 27 \\
3 & 1 & 3 & 18 & 12 & 60 & 72 & 123 & 37 & 629 & 85 & 208 & 598 & 1386 & 146 & 2263 & 146 & 2588 & 1340 & 428 & 174 \\
4 & $-$ & 1 & 7 & 11 & 59 & 68 & 163 & 75 & 1086 & 177 & 422 & 1147 & 4165 & 441 & 5731 & 515 & 7726 & 4008 & 1718 & 878 \\
5 & $-$ & $-$ & 4 & 7 & 44 & 65 & 185 & 115 & 1604 & 305 & 809 & 1956 & 9684 & 1093 & 12238 & 1559 & 19585 & 10969 & 5950 & 3508 \\
6 & $-$ & $-$ & $-$ & 2 & 32 & 47 & 184 & 139 & 2084 & 439 & 1354 & 2971 & 18878 & 2283 & 23383 & 3957 & 44631 & 26864 & 17803 & 11462 \\
7 & $-$ & $-$ & $-$ & $-$ & 20 & 27 & 169 & 138 & 2352 & 549 & 2041 & 4079 & 32496 & 4153 & 41461 & 8666 & 93741 & 58552 & 47038 & 31930 \\
8 & $-$ & $-$ & $-$ & $-$ & 7 & 13 & 132 & 116 & 2320 & 614 & 2828 & 5055 & 50069 & 6753 & 69326 & 16742 & 181948 & 114599 & 112933 & 78261 \\
9 & $-$ & $-$ & $-$ & $-$ & 1 & 5 & 82 & 83 & 2000 & 609 & 3537 & 5746 & 69333 & 10036 & 109094 & 29248 & 326634 & 203239 & 250167 & 173735 \\
10 & $-$ & $-$ & $-$ & $-$ & $-$ & 1 & 38 & 50 & 1473 & 539 & 3917 & 6020 & 86177 & 13705 & 159508 & 46871 & 541098 & 331751 & 514525 & 355837 \\
11 & $-$ & $-$ & $-$ & $-$ & $-$ & $-$ & 12 & 23 & 911 & 410 & 3746 & 5777 & 95517 & 17019 & 213153 & 68842 & 825026 & 503981 & 976094 & 676704 \\
12 & $-$ & $-$ & $-$ & $-$ & $-$ & $-$ & 1 & 7 & 460 & 257 & 3025 & 4932 & 93632 & 18890 & 256311 & 91670 & 1151013 & 713859 & 1691134 & 1189545 \\
13 & $-$ & $-$ & $-$ & $-$ & $-$ & $-$ & $-$ & 1 & 171 & 127 & 2031 & 3642 & 80492 & 18413 & 273869 & 109238 & 1460133 & 936712 & 2648865 & 1914675 \\
14 & $-$ & $-$ & $-$ & $-$ & $-$ & $-$ & $-$ & $-$ & 44 & 47 & 1101 & 2256 & 60195 & 15560 & 257721 & 115210 & 1673574 & 1126073 & 3725223 & 2795661 \\
15 & $-$ & $-$ & $-$ & $-$ & $-$ & $-$ & $-$ & $-$ & 5 & 11 & 471 & 1140 & 38790 & 11276 & 211956 & 106655 & 1723949 & 1226341 & 4683075 & 3677457 \\
16 & $-$ & $-$ & $-$ & $-$ & $-$ & $-$ & $-$ & $-$ & $-$ & 1 & 144 & 453 & 21253 & 6940 & 151375 & 86183 & 1588363 & 1198385 & 5250045 & 4339178 \\
17 & $-$ & $-$ & $-$ & $-$ & $-$ & $-$ & $-$ & $-$ & $-$ & $-$ & 29 & 133 & 9722 & 3583 & 93176 & 60473 & 1302895 & 1043286 & 5240659 & 4580946 \\
18 & $-$ & $-$ & $-$ & $-$ & $-$ & $-$ & $-$ & $-$ & $-$ & $-$ & 2 & 25 & 3634 & 1524 & 49001 & 36624 & 946942 & 804481 & 4652068 & 4319948 \\
19 & $-$ & $-$ & $-$ & $-$ & $-$ & $-$ & $-$ & $-$ & $-$ & $-$ & $-$ & 2 & 1064 & 518 & 21711 & 18997 & 606265 & 546641 & 3666813 & 3633698 \\
20 & $-$ & $-$ & $-$ & $-$ & $-$ & $-$ & $-$ & $-$ & $-$ & $-$ & $-$ & $-$ & 223 & 133 & 7941 & 8344 & 339719 & 325520 & 2560541 & 2721551 \\
21 & $-$ & $-$ & $-$ & $-$ & $-$ & $-$ & $-$ & $-$ & $-$ & $-$ & $-$ & $-$ & 28 & 23 & 2309 & 3043 & 165034 & 168737 & 1579317 & 1810729 \\
22 & $-$ & $-$ & $-$ & $-$ & $-$ & $-$ & $-$ & $-$ & $-$ & $-$ & $-$ & $-$ & 2 & 2 & 501 & 893 & 68723 & 75450 & 856413 & 1066672 \\
23 & $-$ & $-$ & $-$ & $-$ & $-$ & $-$ & $-$ & $-$ & $-$ & $-$ & $-$ & $-$ & $-$ & $-$ & 70 & 200 & 24022 & 28725 & 405904 & 553828 \\
24 & $-$ & $-$ & $-$ & $-$ & $-$ & $-$ & $-$ & $-$ & $-$ & $-$ & $-$ & $-$ & $-$ & $-$ & 4 & 30 & 6862 & 9119 & 166550 & 251863 \\
25 & $-$ & $-$ & $-$ & $-$ & $-$ & $-$ & $-$ & $-$ & $-$ & $-$ & $-$ & $-$ & $-$ & $-$ & $-$ & 2 & 1507 & 2332 & 58439 & 99445 \\
26 & $-$ & $-$ & $-$ & $-$ & $-$ & $-$ & $-$ & $-$ & $-$ & $-$ & $-$ & $-$ & $-$ & $-$ & $-$ & $-$ & 234 & 451 & 17128 & 33657 \\
27 & $-$ & $-$ & $-$ & $-$ & $-$ & $-$ & $-$ & $-$ & $-$ & $-$ & $-$ & $-$ & $-$ & $-$ & $-$ & $-$ & 18 & 57 & 4061 & 9574 \\
28 & $-$ & $-$ & $-$ & $-$ & $-$ & $-$ & $-$ & $-$ & $-$ & $-$ & $-$ & $-$ & $-$ & $-$ & $-$ & $-$ & $-$ & 3 & 721 & 2216 \\
29 & $-$ & $-$ & $-$ & $-$ & $-$ & $-$ & $-$ & $-$ & $-$ & $-$ & $-$ & $-$ & $-$ & $-$ & $-$ & $-$ & $-$ & $-$ & 86 & 394 \\
30 & $-$ & $-$ & $-$ & $-$ & $-$ & $-$ & $-$ & $-$ & $-$ & $-$ & $-$ & $-$ & $-$ & $-$ & $-$ & $-$ & $-$ & $-$ & 4 & 48 \\
31 & $-$ & $-$ & $-$ & $-$ & $-$ & $-$ & $-$ & $-$ & $-$ & $-$ & $-$ & $-$ & $-$ & $-$ & $-$ & $-$ & $-$ & $-$ & $-$ & 3 \\
\hline 
\hline 
$\Sigma$ & 10 & 18 & 64 & 51 & 272 & 365 & 1190 & 807 & 15422 & 4221 & 25792 & 46236 & 677116 & 132543 & 1962756 & 814155 & 13102946 & 9461929 & 39133822 & 34333611 \\
\hline 
\end{tabular}
\caption{The number of connected vertex-transitive graphs of the given order $n$, $8 \leq n \leq 46$ even, and number of edge orbits $o_e$.}
\label{tbl:1full}
\end{sidewaystable}

\subsection{Families with \texorpdfstring{$\boldsymbol{(o_v, o_e) = (1, 2)}$}{(o\_v, o\_e) = (1, 2)}}
\label{subsec:fam12}

From the line for $o_e = 2$ in Table~\ref{tbl:1} it appears likely that vertex transitive nut graphs with two edge orbits exist for all feasible orders.
This is confirmed by the next theorem.

\begin{theorem}
\label{thm:fam1_2}
For every even $n \geq 8$, there exists a nut graph $G$ with $o_v(G) = 1$ and $o_e(G) = 2$.
\end{theorem}

To prove this, we provide three families of quartic vertex-transitive graphs, which together cover all feasible orders and are described in Propositions~\ref{prop:fam1_2_a} to~\ref{prop:fam1_2_c}.
For the first family, let $A_\ell$, where $\ell \geq 3$, be the antiprism on $2\ell$ vertices. Gauci et al.\ \cite{GPS} proved the following proposition.

\begin{proposition}[\cite{GPS}]
\label{prop:fam1_2_a}
The antiprism graph $A_\ell$ of order $2\ell$ is a nut graph if and only if $2\ell \not\equiv 0 \pmod 6$.
\end{proposition}

The next family is composed of cartesian products.\begin{proposition}
\label{prop:fam1_2_b}
The graph $C_3 \cart C_\ell$ of order $3\ell$ is a nut graph for even $\ell \geq 4$ such that $\ell \not\equiv 0 \pmod 6$.
\end{proposition}

\begin{proof}
It is known that $\sigma(G \cart H) = \{ \lambda + \mu \mid \lambda \in \sigma(G) \text{ and } \mu \in \sigma(H) \}$ (see \cite[Section 1.4.6]{haemers}).
Moreover, let $\x_G$ be an eigenvector for an eigenvalue $\lambda \in \sigma(G)$ and let $\x_H$ be an eigenvector for an eigenvalue $\mu \in \sigma(H)$. Then $\x_{G\cart H}$, defined as $\x_{G \cart H}((u, v)) = \x_G(u) \x_H(v)$, is an eigenvector for the eigenvalue $\lambda + \mu$.
It is also well known that $\sigma(C_\ell) = \{ 2\cos(2\pi j / \ell) \mid 0 \leq j < \ell \}$ (see \cite[Section~1.4.3]{haemers}). In particular, $\sigma(C_3) = \{2, -1, -1\}$ and when $\ell$ is even and $\ell \not\equiv 0 \pmod 6$, it is clear that $\sigma(C_\ell)$ contains $-2$ with multiplicity $1$, but not $1$. Therefore,
$C_3 \cart C_\ell$ contains a $0$ eigenvalue with multiplicity $1$. As the eigenvector of $C_\ell$ for the eigenvalue $-2$ is full and so is the eigenvector of $C_3$ for the eigenvalue $2$, it immediately follows that $C_3 \cart C_\ell$ is a nut graph.
\end{proof}

For the third family, a variation on the cartesian product is
used. 
Suppose that vertices of $C_\ell$ are labeled $0, 1, \ldots \ell - 1$ such that $i$ and $i + 1$ are adjacent (indices modulo $\ell$). Then the \emph{ twisted
product} of $C_k$ and $C_\ell$, denoted $C_k \twist C_\ell$, has the vertex set $V(C_k \twist C_\ell) = V(C_k \cart C_\ell)$ and the edge set
$$E(C_k \twist C_\ell) = E(C_k \cart C_\ell) \setminus \{ (i, 0)(i, 1) \mid 0 \leq i < k \} \cup \{ (i, 0)((i + 1) \bmod k, 1) \mid 0 \leq i < k \}.$$
In other words, the construction $C_k \twist C_\ell$, is similar to the cartesian product of $C_k$ and $C_\ell$, but with a twist introduced between the first two $C_k$ layers.

\begin{proposition}
\label{prop:fam1_2_c}
The graph $C_3 \twist C_\ell$ of order $3\ell$ is a nut graph for even $\ell \geq 6$ such that $\ell \equiv 0 \pmod 6$.
\end{proposition}

\begin{proof}
The \emph{twisted product} $C_3 \twist C_\ell$ is an example of a \emph{graph bundle} \cite{Pisanski1982,Pisanski1983}.
Kwak et al.~\cite{Kwak1992,Kwak1993} studied characteristic polynomials of some specific graph bundles \cite{Kwak1992,Kwak1993}.
Here, we apply their Theorem~8 from \cite{Kwak1992}; in the language of  \cite[Theorem~8]{Kwak1992}, our $C_3 \twist C_\ell$ is in fact $C_\ell \times^{\phi} C_n$, where $\phi$ is an $\Aut(C_n)$-voltage assignment and $n = 3$. $\Aut(C_3)$ contains $\mathbb{Z}_3$ as a subgroup. In our case, $\phi$ maps every directed edge of $\vec{C}_\ell$ to $0$ of $\mathbb{Z}_3$, except for the directed edges $(0, 1)$ and $(1, 0)$ which are mapped to $1$ and its inverse $2$, respectively.

Define an $\ell \times \ell$ matrix $M_z$, where $z \in \mathbb{C}$, as follows
$$
(M_z)_{i,j} = \begin{cases}
z, & i = 0 \text{ and } j = 1; \\
\bar{z}, & i = 1 \text{ and } j = 0; \\
1, &(i, j) \notin \{(1, 0), (0,1 )\} \text{ and }  i - j \equiv \pm 1 \pmod \ell \\
0, & \text{otherwise}.
\end{cases}
$$
Note that $M_1$ is the adjacency matrix of $C_\ell$. Let $\omega = (-1 +\sqrt{3}i) / 2$. Theorem 8 from \cite{Kwak1992} gives
$$
\Phi(C_3 \twist C_\ell; \lambda) = \Phi(M_1; \lambda - 2) \cdot \Phi(M_{\omega}; \lambda + 1) \cdot \Phi(M_{\bar{\omega}}; \lambda + 1) = \Phi(M_1; \lambda - 2) \cdot \Phi(M_{\omega}; \lambda + 1)^2.
$$
We show that the nullity of $C_3 \twist C_\ell$ is $1$ for $\ell \equiv 0 \pmod 6$.  First, $\Phi(M_1; \lambda - 2)$ contributes one factor $\lambda$, as
$2$ is an eigenvalue of $C_\ell$ with multiplicity $1$. To show that $\Phi(M_{\omega}; \lambda + 1)$ does not contribute additional factors $\lambda$, we show that $M_{\omega} + I_{\ell \times \ell}$ is of full rank.

Let $B = M_{\omega} + I_{\ell \times \ell}$. Let $B_i$ denote the $i$-th column of $B$. 
From $B$ we can obtain an equivalent matrix $C$ by defining $C_i = B_i - B_{i + 1}$ for $i \geq 1$ (indices modulo $\ell$) and 
$C_0 = B_0 - \omega^2 B_1$. Note that 
$$
C_{ij} = \begin{cases}
\omega, & i = 0 \text{ and } j = 1; \\
1, & (i, j) \neq (0, 1) \text{ and } i + 1 \equiv j \pmod \ell; \\
-\omega^2, & (i, j) \in \{(2, 0), (1, \ell - 1)\} \\
-1, & (i, j) \notin \{(2, 0), (1, \ell - 1)\} \text{ and } i - 2 \equiv j \pmod \ell; \\
0, & \text{otherwise}.
\end{cases}
$$
From matrix $C$ we can obtain an equivalent matrix $D$ by permuting columns, namely
$$D = [ C_1 \ C_4 \ C_7 \ldots C_{\ell - 2} \mid C_2 \ C_5 \ C_8 \ldots C_{\ell - 1} \mid C_0 \ C_3 \ C_6 \ldots C_{\ell - 3} ].$$ 
Note that matrix $D$ is composed of three blocks of size $\ell \times (\ell/3)$. Block $i$, $0 \leq i \leq 2$, contains
non-zero entries only in rows $j$, $0 \leq j < \ell$, such that $j \equiv i \pmod 3$.
Now, we can define a matrix $E$ that is equivalent to matrix $D$ by defining, for $0 \leq i < \ell / 3$,
\begin{align*}
E_i & = \sum_{j=0}^ {\ell / 3 - i - 1} D_j + \sum_{j=\ell / 3 - i}^ {\ell / 3 - 1} \omega D_j; \\
E_{i+\ell / 3} & = \sum_{j=0}^ {i - 1} \omega^2 D_{j+\ell/3} + \sum_{j=i}^ {\ell / 3 - 1} D_{j+\ell/3}; \\
E_{i+2\ell / 3} & = \sum_{j=0}^ {\ell / 3 - i - 1} D_{(j + 1) \bmod (\ell/3) +2\ell/3} + \sum_{j=\ell / 3 - i}^ {\ell / 3 - 1} \omega D_{(j + 1) \bmod (\ell/3) +2\ell/3}. 
\end{align*}
Matrix $E$ has a single non-zero entry in each row and each column; $2\ell/3 - 1$ of these entries are $\omega - 1$ and $\ell/3 + 1$ of these
entries are $\omega + 2$, and the determinant is $(-1)^{\ell/6}(\omega - 1)^{2\ell/3 - 1}(\omega + 2)^{\ell/3 + 1} = 3^{\ell/2} \omega^{\ell/6 + 2}\neq 0$.
But matrix $E$ is equivalent to $B$ which is therefore of full rank. Hence the nullity of $B$ is 1 and therefore $C_3 \twist C_\ell$ is a nut graph.
\end{proof}

Note that the proof does not require explicit construction of the kernel eigenvector. However, it is easily obtained. 
Define $\x\colon V(C_3 \twist C_\ell) \to \mathbb{R}$ by $\x((i, j)) = (-1)^{j}$. Observe that $\x \in \ker \mathbf{A}(C_3 \twist C_\ell)$ and is a full vector.

\begin{figure}[!htb]
\centering
\subcaptionbox{\label{subfig:fam12a}$A_4$}
{ \includegraphics[scale=0.8]{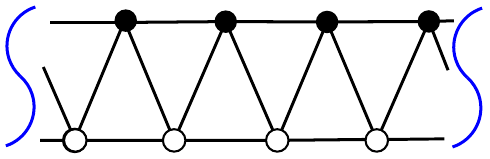} }

\subcaptionbox{\label{subfig:fam12b}$C_3 \cart C_4$}
{ \quad \includegraphics[scale=0.8]{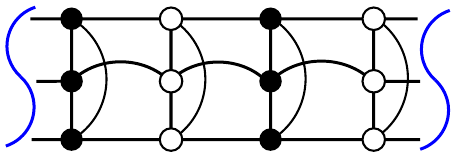} }

\subcaptionbox{\label{subfig:fam12c}$C_3 \twist C_6$}
{ \includegraphics[scale=0.8]{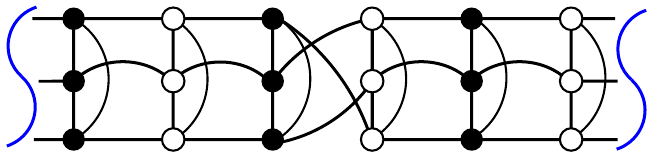} }
\caption{The smallest examples of each of the families described in Propositions~\ref{prop:fam1_2_a} to~\ref{prop:fam1_2_c}. They are shown embedded on a circular strip;
the end blue curves are to be identified. Entries in the kernel eigenvector in each graph are all of equal magnitude and represented by circles colour-coded for sign. }
\label{fig:fam1_2_examples}
\end{figure}

\begin{proof}[Proof of Theorem~\ref{thm:fam1_2}]
By Theorem~\ref{thm:four}, orders and degrees of vertex-transitive nut graphs are even. In fact, our families are all quartic.
The family $A_\ell$, where $\ell \geq 4$ even, described in Proposition~\ref{prop:fam1_2_a}, covers orders $\{ n \geq 8 \mid n \text{ even} \text{ and } n \not\equiv 0 \pmod 6 \}$.
The family $C_3 \cart C_\ell$, where $\ell \geq 4$ even, described in Proposition~\ref{prop:fam1_2_b}, covers orders $\{ n \geq 12 \mid n \equiv 0 \pmod 6 \text{ and } n \not\equiv 0 \pmod{18} \}$.
Finally, the family  $C_3 \twist C_\ell$, where $\ell \geq 6$ even, described in Proposition~\ref{prop:fam1_2_c}, covers orders $\{ n \geq 18 \mid n \equiv 0 \pmod{18} \}$.
\end{proof}

There exist vertex transitive nut graphs that are not Cayley graphs. The three minimal examples of non-Cayley nut graphs have order $16$, with invariants $(d(G), o_e(G), |\mathrm{Aut}(G)|)$ of $(4, 3, 32)$, $(6, 4, 32)$ and $(10, 5, 32)$, respectively. The quartic example is shown in Figure~\ref{fig:nonCayleyExamples4}(a); it is a tetracirculant with vertex set 
$\{ u_i, v_i, w_i, z_i \mid i \in \mathbb{Z}_4 \}$ and edge set $\{ u_i v_i, v_i w_i, u_i z_i, u_i u_{i+1}, v_i v_{i+1}, z_i w_{i+1}, z_i w_{i+2}, z_i w_{i+3}  \mid i \in \mathbb{Z}_4 \}$.
The second smallest quartic example is shown in Figure~\ref{fig:nonCayleyExamples4}(b) 
and is one of $14$ non-Cayley nut graphs of order $30$; it has $2$ edge orbits and its automorphism group is of order $120$. This is a generalisation of Rose Window graphs; its vertex set is $\{u_i, v_i \mid 
i \in \mathbb{Z}_{15}\}$ and its edge set is $\{u_i v_i, u_i v_{i+5}, v_i v_{i+3}, u_i u_{i+6}    \mid i \in \mathbb{Z}_{15} \}$.

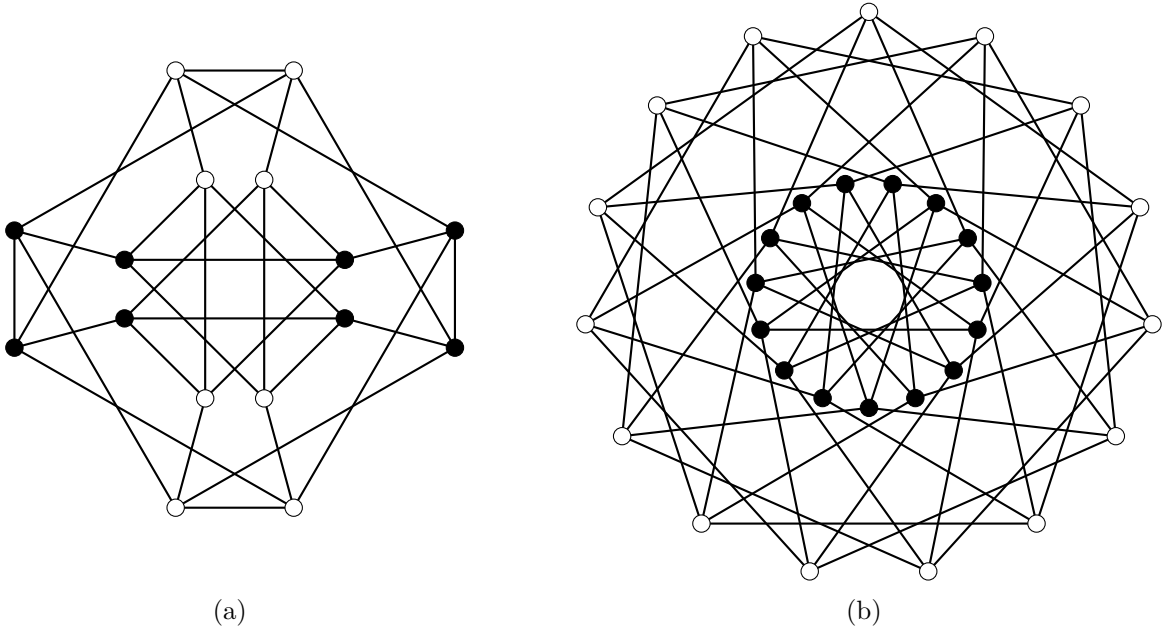
\begin{figure}[!htb]
\centering
\subcaptionbox{\label{subfig:nonCay15}} { 
\begin{tikzpicture}[scale=1.5]
\tikzstyle{edge}=[draw,thick]
\tikzstyle{every node}=[draw, circle, fill=blue!50!white, inner sep=2.3pt]
\foreach \i in {0,1,...,3}{
 \coordinate (a\i) at ({\i * 360 / 4 + 30 - 15}:2.0);
 \coordinate (b\i) at ({\i * 360 / 4 - 15}:2.0);
 \coordinate (c\i) at ({\i * 360 / 4 - 15}:1.0);
 \coordinate (d\i) at ({\i * 360 / 4 + 15}:1.0);
 }
\foreach \i in {0,1,...,3}{
	\pgfmathtruncatemacro{\j}{mod(\i + 1, 4)}
	\pgfmathtruncatemacro{\jj}{mod(\i + 2, 4)}
	\pgfmathtruncatemacro{\jjj}{mod(\i + 3, 4)}
	\draw[edge] (a\i) -- (b\i);
	\draw[edge] (b\i) -- (c\i);
	\draw[edge] (a\i) -- (d\i);
	\draw[edge] (a\i) -- (a\j);
	\draw[edge] (b\i) -- (b\j);
	\draw[edge] (d\i) -- (c\j);
	\draw[edge] (d\i) -- (c\jj);
	\draw[edge] (d\i) -- (c\jjj);
}
\node[fill=none,draw=none] at (0, -2.5) {};
\foreach \i/\c in {0/black,1/white,2/black,3/white} {
  \node[fill=\c] at (b\i) {};
  \node[fill=\c] at (a\i) {};
 \node[fill=\c] at (c\i) {};
  \node[fill=\c] at (d\i) {};
}
\end{tikzpicture}}
\qquad\quad
\subcaptionbox{\label{subfig:nonCay30}} { 
\begin{tikzpicture}[scale=1.5]
\tikzstyle{edge}=[draw,thick]
\tikzstyle{every node}=[draw, circle, fill=blue!50!white, inner sep=2.3pt]
\foreach \i in {0,1,...,14}{
  \coordinate (u\i) at ({\i * 360 / 15 + 60 + 90}:1.0); 
 }
\foreach \i in {0,1,...,14}{
  \coordinate (v\i) at ({\i * 360 / 15 + 90}:2.5);
 }
\foreach \i in {0,1,...,14}{
	\pgfmathtruncatemacro{\j}{mod(\i + 5, 15)}
	\pgfmathtruncatemacro{\k}{mod(\i + 3, 15)}
	\pgfmathtruncatemacro{\kk}{mod(\i + 6, 15)}
	\draw[edge] (u\i) -- (v\i);
	\draw[edge] (u\i) -- (v\j);
	\draw[edge] (v\i) -- (v\k);
	\draw[edge] (u\i) -- (u\kk);
}
\foreach \i in {0,1,...,14}{
  \node[fill=black] at (u\i) {};
  \node[fill=white] at (v\i) {};
 }
\end{tikzpicture}
}
\caption{The two smallest 4-valent non-Cayley vertex transitive nut graphs.
Entries in the kernel eigenvector in each graph are all of equal magnitude and represented by circles colour-coded for sign.}
\label{fig:nonCayleyExamples4}
\end{figure}

\section{Nut graph with two vertex orbits}

Data are available for graphs with two vertex orbits \cite{gordonCensus} and Table~\ref{tbl:2} shows our
analysis for small graphs of this class.
We observe that $(o_v(G), o_e(G)) = (2, 1)$ and $(o_v(G), o_e(G)) = (2, 2)$ do not
occur in the table. This observation is, of course, consistent with Theorem~\ref{thm:norb} from Section~\ref{sec:generalResult}.
We also observe that
the number of edge orbits can be large.
Here, we provide infinite families of nut graphs with two vertex orbits and three edge orbits.

\begin{table}[!ht]
\centering
\begin{tabular}{|c||rrrrrrrrrrrrrr|}
\hline 
\backslashbox{$o_e$}{$n$}
 & 9 & 10 & 12 & 14 & 15 & 16 & 18 & 20 & 21 & 22 & 24 & 25 & 26 & 27 \\
\hline 
\hline 
1 & 0 & 0 & 0 & 0 & 0 & 0 & 0 & 0 & 0 & 0 & 0 & 0 & 0 & 0 \\
2 & 0 & 0 & 0 & 0 & 0 & 0 & 0 & 0 & 0 & 0 & 0 & 0 & 0 & 0 \\
3 & 1 & 1 & 4 & 7 & 6 & 7 & 10 & 20 & 10 & 19 & 33 & 13 & 26 & 19 \\
4 & 0 & 3 & 6 & 2 & 16 & 12 & 16 & 72 & 62 & 6 & 169 & 46 & 19 & 124 \\
5 & 0 & 0 & 12 & 1 & 5 & 24 & 78 & 133 & 40 & 20 & 665 & 66 & 44 & 90 \\
6 & $-$ & 0 & 6 & 3 & 6 & 31 & 99 & 134 & 48 & 122 & 1460 & 160 & 327 & 227 \\
7 & $-$ & $-$ & 5 & 1 & 3 & 31 & 133 & 171 & 77 & 94 & 3418 & 191 & 348 & 445 \\
8 & $-$ & $-$ & 4 & 1 & 0 & 78 & 102 & 310 & 77 & 110 & 7031 & 234 & 552 & 671 \\
9 & $-$ & $-$ & 1 & 0 & 0 & 53 & 136 & 264 & 40 & 184 & 12081 & 429 & 1118 & 777 \\
10 & $-$ & $-$ & 0 & $-$ & 1 & 80 & 71 & 381 & 88 & 45 & 19694 & 599 & 283 & 1984 \\
11 & $-$ & $-$ & $-$ & $-$ & 0 & 73 & 82 & 392 & 193 & 14 & 28013 & 156 & 340 & 5192 \\
12 & $-$ & $-$ & $-$ & $-$ & $-$ & 49 & 18 & 366 & 4 & 154 & 36902 & 574 & 2258 & 797 \\
13 & $-$ & $-$ & $-$ & $-$ & $-$ & 17 & 20 & 165 & 49 & 0 & 41123 & 267 & 77 & 3996 \\
14 & $-$ & $-$ & $-$ & $-$ & $-$ & 13 & 2 & 147 & 0 & 0 & 44395 & 8 & 4 & 292 \\
15 & $-$ & $-$ & $-$ & $-$ & $-$ & 3 & 0 & 238 & 0 & 0 & 39101 & 1 & 15 & 261 \\
16 & $-$ & $-$ & $-$ & $-$ & $-$ & 0 & $-$ & 52 & 0 & 10 & 36325 & 0 & 735 & 420 \\
17 & $-$ & $-$ & $-$ & $-$ & $-$ & $-$ & $-$ & 9 & 0 & 0 & 24477 & 0 & 0 & 1239 \\
18 & $-$ & $-$ & $-$ & $-$ & $-$ & $-$ & $-$ & 18 & $-$ & $-$ & 19068 & 2 & 0 & 136 \\
19 & $-$ & $-$ & $-$ & $-$ & $-$ & $-$ & $-$ & 1 & $-$ & $-$ & 8568 & 2 & 0 & 171 \\
20 & $-$ & $-$ & $-$ & $-$ & $-$ & $-$ & $-$ & 0 & $-$ & $-$ & 5638 & $-$ & 20 & 0  \\
21 & $-$ & $-$ & $-$ & $-$ & $-$ & $-$ & $-$ & $-$ & $-$ & $-$ & 2173 & $-$ & 0 & 0  \\
22 & $-$ & $-$ & $-$ & $-$ & $-$ & $-$ & $-$ & $-$ & $-$ & $-$ & 838 & $-$ & $-$ & 0  \\
23 & $-$ & $-$ & $-$ & $-$ & $-$ & $-$ & $-$ & $-$ & $-$ & $-$ & 140 & $-$ & $-$ & 0  \\
24 & $-$ & $-$ & $-$ & $-$ & $-$ & $-$ & $-$ & $-$ & $-$ & $-$ & 63 & $-$ & $-$ & $-$  \\
25 & $-$ & $-$ & $-$ & $-$ & $-$ & $-$ & $-$ & $-$ & $-$ & $-$ & 7 & $-$ & $-$ & $-$  \\
26 & $-$ & $-$ & $-$ & $-$ & $-$ & $-$ & $-$ & $-$ & $-$ & $-$ & 0 & $-$ & $-$ & $-$  \\
\hline 
\hline 
$\Sigma$ & 1 & 4 & 38 & 15 & 37 &  471 & 767 & 2873 & 688 & 778 & 331382 & 2748 & 6166 & 16841  \\
\hline 
\end{tabular}
\caption{The number of nut graphs with precisely two vertex orbits of the given order $n$ and number of edge orbits $o_e$.}
\label{tbl:2}
\end{table}

\begin{sidewaystable}[!htbp]
\centering
\begin{tabular}{|c||rrrrrrrrrrrrrr|}
\hline 
\backslashbox{$o_e$}{$n$}
 & 9 & 10 & 12 & 14 & 15 & 16 & 18 & 20 & 21 & 22 & 24 & 25 & 26 & 27 \\
\hline 
\hline 
1 & 5 & 5 & 8 & 8 & 15 & 11 & 14 & 21 & 24 & 16 & 31 & 23 & 18 & 30 \\
2 & 29 & 43 & 98 & 103 & 151 & 190 & 285 & 420 & 341 & 315 & 869 & 433 & 449 & 628 \\
3 & 34 & 74 & 270 & 305 & 402 & 718 & 1341 & 2117 & 1332 & 1624 & 6279 & 1961 & 2968 & 3743 \\
4 & 12 & 52 & 331 & 363 & 514 & 1352 & 2903 & 5318 & 2573 & 3621 & 22524 & 4379 & 8593 & 10968 \\
5 & 4 & 17 & 284 & 258 & 469 & 1738 & 4359 & 9211 & 3725 & 5488 & 56544 & 7739 & 16838 & 23380 \\
6 & $-$ & 1 & 183 & 129 & 345 & 1879 & 5130 & 12453 & 4459 & 6462 & 112054 & 11823 & 26114 & 41411 \\
7 & $-$ & $-$ & 110 & 50 & 251 & 1831 & 5496 & 14313 & 4999 & 6614 & 190905 & 16078 & 35084 & 66769 \\
8 & $-$ & $-$ & 53 & 13 & 152 & 1787 & 5305 & 14885 & 5255 & 6056 & 292831 & 19694 & 42043 & 99346 \\
9 & $-$ & $-$ & 22 & 2 & 87 & 1627 & 4714 & 14377 & 5316 & 4993 & 416618 & 22044 & 45677 & 139401 \\
10 & $-$ & $-$ & 3 & $-$ & 30 & 1427 & 3597 & 13039 & 4845 & 3683 & 555666 & 22452 & 45324 & 182434 \\
11 & $-$ & $-$ & $-$ & $-$ & 8 & 1086 & 2365 & 11191 & 3992 & 2411 & 694869 & 20502 & 40992 & 221082 \\
12 & $-$ & $-$ & $-$ & $-$ & $-$ & 734 & 1213 & 9054 & 2763 & 1406 & 809588 & 16446 & 33697 & 242741 \\
13 & $-$ & $-$ & $-$ & $-$ & $-$ & 392 & 499 & 6826 & 1626 & 721 & 872753 & 11346 & 25145 & 239567 \\
14 & $-$ & $-$ & $-$ & $-$ & $-$ & 169 & 128 & 4666 & 747 & 318 & 863949 & 6566 & 16956 & 208850 \\
15 & $-$ & $-$ & $-$ & $-$ & $-$ & 49 & 21 & 2832 & 277 & 117 & 780210 & 3118 & 10260 & 160119 \\
16 & $-$ & $-$ & $-$ & $-$ & $-$ & 9 & $-$ & 1457 & 66 & 30 & 639027 & 1179 & 5482 & 106253 \\
17 & $-$ & $-$ & $-$ & $-$ & $-$ & $-$ & $-$ & 624 & 12 & 5 & 471486 & 340 & 2524 & 60839 \\
18 & $-$ & $-$ & $-$ & $-$ & $-$ & $-$ & $-$ & 204 & $-$ & $-$ & 311318 & 69 & 959 & 29385 \\
19 & $-$ & $-$ & $-$ & $-$ & $-$ & $-$ & $-$ & 48 & $-$ & $-$ & 182116 & 8 & 288 & 11915 \\
20 & $-$ & $-$ & $-$ & $-$ & $-$ & $-$ & $-$ & 6 & $-$ & $-$ & 93435 & $-$ & 60 & 3857 \\
21 & $-$ & $-$ & $-$ & $-$ & $-$ & $-$ & $-$ & $-$ & $-$ & $-$ & 41330 & $-$ & 8 & 993 \\
22 & $-$ & $-$ & $-$ & $-$ & $-$ & $-$ & $-$ & $-$ & $-$ & $-$ & 15463 & $-$ & $-$ & 173 \\
23 & $-$ & $-$ & $-$ & $-$ & $-$ & $-$ & $-$ & $-$ & $-$ & $-$ & 4716 & $-$ & $-$ & 21 \\
24 & $-$ & $-$ & $-$ & $-$ & $-$ & $-$ & $-$ & $-$ & $-$ & $-$ & 1120 & $-$ & $-$ & $-$ \\
25 & $-$ & $-$ & $-$ & $-$ & $-$ & $-$ & $-$ & $-$ & $-$ & $-$ & 180 & $-$ & $-$ & $-$ \\
26 & $-$ & $-$ & $-$ & $-$ & $-$ & $-$ & $-$ & $-$ & $-$ & $-$ & 14 & $-$ & $-$ & $-$ \\
\hline 
\hline 
$\Sigma$ & 84 & 192 & 1362 & 1231 & 2424 & 14999 & 37370 & 123062 & 42352 & 43880 & 7435895 & 166200 & 359479 & 1853905 \\
\hline 
\end{tabular}
\caption{The number of connected graphs with precisely two vertex orbits of the given order $n$ and number of edge orbits $o_e$.
Only orders where nut graphs with two vertex orbits exist are included.
}
\label{tbl:2full}
\end{sidewaystable}

\subsection{Families with \texorpdfstring{$\boldsymbol{(o_v, o_e) = (2, 3)}$}{(o\_v, o\_e) = (2, 3)}}
\label{subsec:fam23}

From the line for $o_e = 3$ in Table~\ref{tbl:2} it appears that nut graphs with two vertex orbits and three edge orbits exist for all orders $n 
\geq 9$ such that $n$ is not a prime; see Conjecture~\ref{conj:conjecturev2e3}.
Here, we provide two families of such nut graphs; one that covers orders that are multiples of three, and one that covers
orders that are multiples of two but not multiples of three.

\begin{proposition} 
\label{prop:triangCyc}
Let $\mathcal{T}_n$ be a $n$-cycle with a triangle fused to every vertex (see Figure~\ref{fig:families23examples}(a) for an example).
The graph $\mathcal{T}_n$ is a nut graph for every $n \geq 3$.
\end{proposition}

\begin{proof}
Let the vertices of the $n$-cycle be labeled $0, 1, \ldots, n-1$.
Let $a_0, a_1, \ldots, a_{n - 1}$ denote the entries on the vertices of the $n$-cycle in a kernel eigenvector 
of $\mathcal{T}_n$.
It is easy to see that both vertices of the triangle fused to vertex $i$ of the cycle must then carry entry $-a_i$.
The local condition at vertices of the cycle is
\begin{equation}
a_{i - 1} -2 a_{i}  + a_{i + 1} = 0 \qquad \text{for } i = 0, \ldots, n - 1,
\label{eq:localTriangCyc}
\end{equation}
where indices are modulo $n$. Equation \eqref{eq:localTriangCyc} in matrix form is
\begin{equation}
A(C_n) \x = 2 \x,
\label{eq:localTriangCycMat}
\end{equation}
where $A(C_n)$
is the adjacency matrix of the $n$-cycle and $\x =  \begin{bmatrix} a_0 & a_1 & a_2 & \ldots & a_{n - 1}\end{bmatrix}$. 
The cycle $C_n$ is a $2$-regular connected graph, and thus has a unique eigenvalue $2$
in its spectrum, with $\x = \begin{bmatrix} 1 & 1 & 1 & \ldots & 1\end{bmatrix}$ and the solution to 
Equation \eqref{eq:localTriangCyc} is $a_0 = a_1 = \cdots = a_{n - 1} = 1$.
\end{proof}

\begin{figure}[!htb]
\centering
\begin{subfigure}{0.3\textwidth}
\centering
    \includegraphics{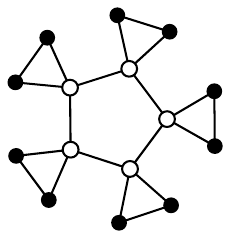}
    \caption{$\mathcal{T}_5$}
    \label{fig:fam23exa}
\end{subfigure}%
\begin{subfigure}{0.3\textwidth}
\centering
   \includegraphics{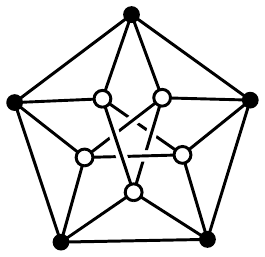}
    \caption{$R_5(1, 2)$}
    \label{fig:fam23exb1}
\end{subfigure}
\begin{subfigure}{0.3\textwidth}
\centering
   \includegraphics{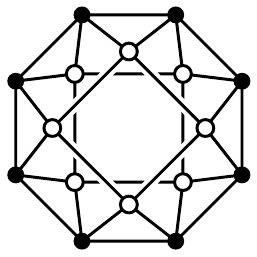}
    \caption{$R_8(1, 2)$}
    \label{fig:fam23exb2}
\end{subfigure}
\caption{Small examples of each of the families described in Propositions~\ref{prop:triangCyc} and~\ref{prop:roseWindow}.
Entries in the kernel eigenvector in each graph are all of equal magnitude and represented by circles colour-coded for sign.}
\label{fig:families23examples}
\end{figure}

In 2008, the family of \emph{Rose Window} graphs was introduced \cite{Stitch2008}. A Rose Window graph, denoted $R_n(a, r)$,
is defined by 
\begin{align*}
V(R_n(a, r)) & = \{ v_0, v_1, \ldots ,v_{n - 1} \} \cup \{ u_0, u_1, \ldots , u_{n - 1} \} \text{ and} \\
E(R_n(a, r)) & = \{ v_i v_{i+1}, u_i u_{i+r} \mid i = 0, \ldots, n - 1\} \cup \{ u_i v_i, u_i v_{i+a} \mid i = 0, \ldots, n - 1\},
\end{align*}
where all indices are modulo $n$. We will consider the subset with $a = 1$ and $r = 2$ (see Figures~\ref{fig:families23examples}(b) and~\ref{fig:families23examples}(c) for examples).

\begin{proposition}
\label{prop:roseWindow}
Let $n \geq 5$. The graph $R_n(1, 2)$ is a core graph for all $n \geq 5$.
The graph $R_n(1, 2)$ is a nut graph if and only if $n \not\equiv 0 \pmod{3}$.
\end{proposition}

\begin{proof}
Let $\x \in \ker \mathbf{A}(R_n(1, 2))$ and let $a_0 = \x(v_0), a_1 = \x(v_1), b_{-2} = \x(u_{n - 2}), b_{-1} = \x(u_{n - 1}), b_0 = \x(u_0)$ and $b_1 = \x(u_1)$.
See Figure~\ref{fig:RoseBud} for an illustration.
\begin{figure}[!htb]
\centering
\includegraphics[scale=0.5]{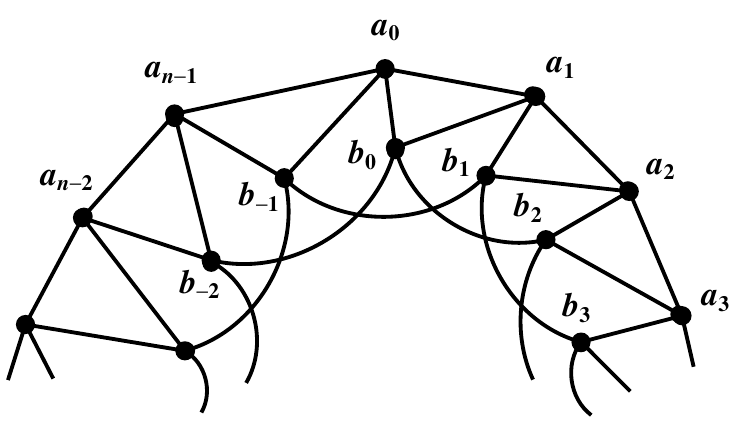}
\caption{Labelling scheme for the Rose Window graph $R_n(1,2)$. Entries in the candidate kernel 
eigenvector are $a_i$ on vertices $v_i$ and $b_i$ on vertices $u_i$, all indices taken modulo $n$.}
\label{fig:RoseBud}
\end{figure}
Using the local condition \eqref{eq:localCondition} at vertices $v_1, \ldots, v_{n - 2}$ and $u_0, \ldots, u_{n - 3}$ the entries in $\x$ of vertices $v_2, \ldots, v_{n-1}$
and $u_2, \ldots, u_{n - 1}$ can be expressed as linear combinations of $a_0, a_1, b_{-2}, b_{-1}, b_0$ and $b_1$. Namely,
\begin{equation}
\label{eq:recurrence}
\begin{aligned}
a_i & = -a_{i-2} - b_{i-2} - b_{i - 1}, \qquad (2 \leq i < n) \\ 
b_i & = -a_{i-2} - a_{i-1} - b_{i - 4}, \qquad (2 \leq i < n) 
\end{aligned}
\end{equation}
where $a_i = \x(v_i)$ and $b_i = \x(u_i)$.
Every entry $\x(v)$, $v \in V(R_n(1, 2)))$, can be assigned a row vector $\boldsymbol{\xi}(v) \in \mathbb{R}^6$, acting as proxy for 
$\x(v) = \boldsymbol{\xi}(v) \cdot [a_0\ a_1\ b_{-2}\ b_{-1}\ b_0\ b_1]$. Solving the linear recurrence relations \eqref{eq:recurrence} 
we obtain
\begin{equation}
a_k = \begin{cases}
\begin{bmatrix} 1 & 0 & \frac{k}{3} & 0 & 0 & -\frac{k}{3} \end{bmatrix}, & k \equiv 0 \pmod{3}; \\[0.5em]
\begin{bmatrix} \frac{k - 1}{3} & \frac{k + 2}{3} & \frac{k - 1}{3} & \frac{k - 1}{3} & 0 & 0 \end{bmatrix}, & k \equiv 1 \pmod{3}; \\[0.5em]
\begin{bmatrix} -\frac{k + 1}{3} & \frac{k - 2}{3} & 0 & \frac{k - 2}{3} & -1 & -\frac{k+1}{3} \end{bmatrix}, & k \equiv 2 \pmod{3}; 
\end{cases}
\end{equation}
and
\begin{equation}
b_k = \begin{cases}
\begin{bmatrix} \frac{k}{3} & -\frac{k}{3} & 0 & -\frac{k}{3} & 1 & \frac{k}{3} \end{bmatrix}, & k \equiv 0 \pmod{3}; \\[0.5em]
\begin{bmatrix} 0 & 0 & -\frac{k - 1}{3} & 0 & 0 & \frac{k + 2}{3} \end{bmatrix}, & k \equiv 1 \pmod{3}; \\[0.5em]
\begin{bmatrix} -\frac{k + 1}{3} & -\frac{k + 1}{3} &  -\frac{k + 1}{3}  & -\frac{k - 2}{3} & 0 & 0 \end{bmatrix}, & k \equiv 2 \pmod{3}. 
\end{cases}
\end{equation}
By using the local condition~\eqref{eq:localCondition} at vertices $v_0, v_{n - 1}, u_{n - 2}, u_{n - 1}$ we obtain the four linear equations
\begin{equation}
\label{eq:roseEq1}
\begin{aligned}
a_{n - 1} + b_{-1} + b_0 + a_1 & = 0, \\
a_{n - 2} + b_{-2} + b_{-1} + a_0 & = 0, \\
a_{n -1} + a_{0} + b_{n-3} + b_1 & = 0, \\
a_{n -2} + a_{n - 1} + b_{n-4} + b_0 & = 0,
\end{aligned}
\end{equation}
relating $a_0, a_1, b_{-2}, b_{-1}, b_0$ and $b_1$ to each other. Two more equations can be obtained from the fact that $b_{n-2} = \boldsymbol{\xi}(u_{n-2}) = b_{-2}$ and
$b_{n-1} = \boldsymbol{\xi}(u_{n-1}) = b_{-1}$:
\begin{equation}
\label{eq:roseEq2}
\begin{aligned}
b_{n-2} - b_{-2} = 0, \\
b_{n-1} - b_{-1} = 0, 
\end{aligned}
\end{equation}
There are three cases to consider. 

Case $n \equiv 0 \pmod{3}$: The equations~\eqref{eq:roseEq1} and~\eqref{eq:roseEq2} can be written in matrix form
\begin{equation}
\label{eq:RoseCase1}
\begin{bmatrix*}[r]
-\mu & \mu & 0 & \mu & 0 & -\mu \\
\mu & \mu & \mu & \mu & 0 & 0 \\
0 & 0 & 0 & 0 & 0 & 0 \\
-\mu & \mu & 0 & \mu & 0 & -\mu \\
0 & 0 & -\mu & 0 & 0 & \mu \\
-\mu & -\mu & -\mu & -\mu & 0 & 0 
\end{bmatrix*}
\begin{bmatrix*}[l]
a_0 \\
a_1 \\
b_{-2} \\
b_{-1} \\
b_0 \\
b_1 
\end{bmatrix*} = \mathbf{0}_{6 \times 1},
\end{equation}
where $\mu = \tfrac{n}{3}$.
It is easy to see that the matrix in \eqref{eq:RoseCase1} is of rank $3$. This implies that $R_n(1, 2)$ has nullity $3$.

Case $n \equiv 1 \pmod{3}$: The equations~\eqref{eq:roseEq1} and~\eqref{eq:roseEq2} can be written in matrix form
\begin{equation}
\label{eq:RoseCase2}
\begin{bmatrix*}[c]
1 & 1 & \mu & 1 & 1 & -\mu \\
1-\mu  & \mu - 1 & 1 & \mu & -1 & -\mu \\
2 & 0 & 1 & 0 & 0 & 1 \\
0 & 0 & \mu & 0 & 1 & -\mu - 1 \\
-\mu & -\mu & -\mu - 1 & 1-\mu & 0 & 0 \\
\mu & -\mu & 0 & -\mu - 1 & 1 & \mu  
\end{bmatrix*}
\begin{bmatrix*}[l]
a_0 \\
a_1 \\
b_{-2} \\
b_{-1} \\
b_0 \\
b_1 
\end{bmatrix*} = \mathbf{0}_{6 \times 1},
\end{equation}
where $\mu = \frac{n-1}{3}$. Note that $\mu \geq 2$ as $n \geq 5$.
Using tedious but elementary linear algebra, the matrix in \eqref{eq:RoseCase2} can be reduced to its echelon form, from which it can be seen that it is of rank $5$. 
This implies that $R_n(1, 2)$ has nullity $1$.

Case $n \equiv 2 \pmod{3}$: The equations~\eqref{eq:roseEq1} and~\eqref{eq:roseEq2} can be written in matrix form
\begin{equation}
\label{eq:RoseCase3}
\begin{bmatrix*}[c]
\mu & \mu + 2 & \mu & \mu + 1 & 1 & 0 \\
2 & 0 & \mu + 1 & 1 & 0 & -\mu  \\
1 & 1 & 0 & 1 & 0 & 1 \\
\mu + 1 & \mu + 1 & \mu + 1 & \mu & 1 & 0 \\
\mu & -\mu & -1 & -\mu & 1 & \mu \\
0 & 0 & -\mu & -1 & 0 & \mu + 1 
\end{bmatrix*}
\begin{bmatrix*}[l]
a_0 \\
a_1 \\
b_{-2} \\
b_{-1} \\
b_0 \\
b_1 
\end{bmatrix*} = \mathbf{0}_{6 \times 1},
\end{equation}
where $\mu = \frac{n-2}{3}$. Note that $\mu \geq 1$ as $n \geq 5$. As before, the matrix in \eqref{eq:RoseCase3} can be reduced to its 
echelon form, from which it can be seen that it is of rank $5$. This implies that $R_n(1, 2)$ has nullity $1$ also in the present case.

It is easily seen that there exists a full vector in $\ker \mathbf{A}(R_n(1, 2))$ in all three cases. Simply take $a_i = 1$ and $b_i = -1$ for all $i$,
hence $R_n(1, 2)$ is a nut graph if $n \not\equiv 0 \pmod{3}$ and 
merely a core graph if $n \equiv 0 \pmod{3}$.
\end{proof}

The graph $R_n(1, 2))$, for $n \equiv 0 \pmod{3}$, has nullity $3$. Possible choices of basis for the nullspace are depicted in Figure~\ref{fig:ortoBasisRose}.
\begin{figure}[!htb]
\centering
\begin{subfigure}{0.25\textwidth}
\centering
    \includegraphics[scale=0.8]{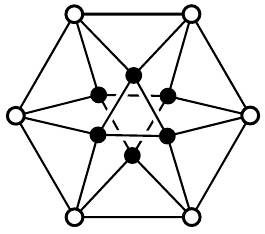}
    \caption{}
    \label{fig:rose6eigv1}
\end{subfigure}%
\begin{subfigure}{0.25\textwidth}
\centering
    \includegraphics[scale=0.8]{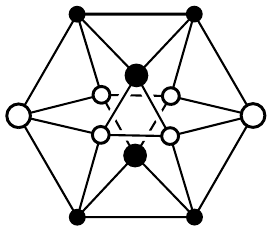}
    \caption{}
    \label{fig:rose6eigv2}
\end{subfigure}%
\begin{subfigure}{0.25\textwidth}
\centering
    \includegraphics[scale=0.8]{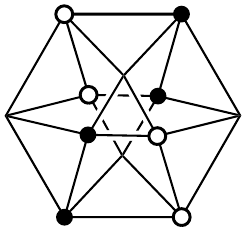}
    \caption{}
    \label{fig:rose6eigv3}
\end{subfigure}%
\begin{subfigure}{0.25\textwidth}
\centering
    \includegraphics[scale=0.8]{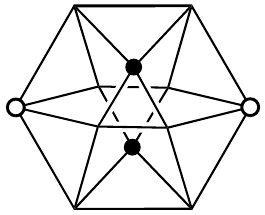}
    \caption{}
    \label{fig:rose6eigv4}
\end{subfigure}%
\caption{Kernel eigenvectors for the graph $R_6(1, 2)$. Vectors (a) to (c) form an orthogonal basis that includes the rotationally symmetric vector that is present in the nullspace 
of $R_n(1, 2)$ for every $n \geq 5$.
An alternative basis consists of vector (d) and its rotations by $\pm 60\degree$. Signs of eigenvector entries are indicated by colour and relative magnitudes by area of the circles, 
where the possible magnitudes are $0$, $1$ and~$2$.}
\label{fig:ortoBasisRose}
\end{figure}

\subsection{Multiplier constructions}
\label{sec:multiConstructions}

The family $\mathcal{T}_n$ can in fact be substantially generalised by defining the \emph{triangle-multiplier construction}. Unlike
some constructions in the literature~\cite{GPS}, the triangle-multiplier applies to parent graphs that are not necessarily nut graphs.

\begin{proposition}
\label{prop:multiplier3}
Let $G$ be a connected $(2t)$-regular graph, where $t \geq 1$. Let $\mathcal{M}_3(G)$ be the graph obtained from $G$ by
fusing a bouquet of $t$ triangles to every vertex of $G$. Then $\mathcal{M}_3(G)$ is a nut graph.
\end{proposition}

\begin{proof}
We follow the pattern of the proof of Proposition~\ref{prop:triangCyc}, where Equation~\eqref{eq:localTriangCyc}
is replaced by
\begin{equation}
\sum_{u \in N(v)} \x(u)   -2t \, \x(v)  = 0 \qquad \text{for } v \in V(G),
\label{eq:localTriang}
\end{equation}
which in matrix form is $A(G) \x = 2t\,\x$. So, kernel eigenvectors of $\mathcal{M}_3(G)$ are precisely eigenvectors of $G$ for the eigenvalue $2t$.
But the graph $G$ is $2t$-regular, so the solution of Equation~\ref{eq:localTriang}  is unique and $\x$, i.e.\ the Perron eigenvector, is full. 
\end{proof}

The choice of name for the construction is justified by the fact that $|V(\mathcal{M}_3(G))| = (2t+1)|V(G)|$. 
As Proposition~\ref{prop:multiConstrSym} in Section~\ref{sec:symmetrySec} will show, the triangle-multiplier construction 
adds one vertex orbit and two edge orbits
to the graph $G$, irrespective of the value $t$.
We can define a
\emph{pentagon-multiplier construction} as follows. As in the case of the triangle-multiplier, this construction applies
to graphs that are not necessarily nut graphs.

\begin{proposition}
\label{prop:multiplier5}
Let $G$ be a bipartite connected $(2p)$-regular graph, where $p \geq 1$. Let $\mathcal{M}_5(G)$ be the graph obtained from $G$ by
fusing a bouquet of $p$ pentagons (i.e.\ $5$-cycles) to every vertex of $G$. Then $\mathcal{M}_5(G)$ is a nut graph.
\end{proposition}

\begin{proof}
We follow the pattern of the proof of Proposition~\ref{prop:multiplier3}. Consider a pentagon fused at a vertex $v \in V(G)$.
The vertices of the pentagon that are adjacent to $v$ both carry entry $+\x(v)$, while the remaining two vertices carry entry $-\x(v)$.

Equation~\eqref{eq:localTriang}
is replaced by
\begin{equation}
\sum_{u \in N(v)} \x(u)   +2p \, \x(v)  = 0 \qquad \text{for } v \in V(G),
\end{equation}
which in matrix form is $A(G) \x = -2p\,\x$. Since $G$ is connected, bipartite and $(2p)$-regular, it has a unique eigenvalue
$-2p$ in its spectrum, and the corresponding eigenvector is full.
\end{proof}

We note that in Proposition~\ref{prop:multiplier3}, any fused triangle could be replaced by a $(4q + 3)$-cycle for any $q \geq 0$.
Likewise, in Proposition~\ref{prop:multiplier5}, any fused pentagon may be replaced by a $(4q + 5)$-cycle for any $q \geq 0$.
In fact, these changes are just repeated applications of the subdivision construction on the triangles (resp.\ pentagons) of the graph
$\mathcal{M}_3(G)$ (resp.\ $\mathcal{M}_5(G)$). 

It seems natural to ask, what would happen if we fuse a mixture of triangles and pentagons to some vertices of a graph? Consider
the case where we fuse a triangle and a pentagon to a vertex in a graph.

\begin{proposition}
\label{prop:specSciriha}
Let $G$ be a nut graph and let $v \in V(G)$ be a vertex. Let $\mathcal{P}(G, v)$ be the graph obtained from $G$ by
fusing a triangle and a pentagon to vertex $v$. Then $\mathcal{P}(G, v)$ is a nut graph.
\end{proposition}

\begin{proof}
 Let $\x$ be a kernel eigenvector of $\mathcal{P}(G, v)$. 
Consider the two vertices on the fused triangle that are adjacent to $v$. Their entries in $\x$ are $-\x(v)$.
Now consider the two vertices of the fused pentagon that are adjacent to $v$. Their entries in $\x$ are $\x(v)$;
the entries of the remaining two vertices are $-\x(v)$. The local condition at vertex $v$ is simply $\sum_{u \in N(v)} \x(u) = 0$.
This means that $\eta(\mathcal{P}(G, v)) = \eta(G)$. Thus, $\mathcal{P}(G, v)$ is a nut graph if and only if $G$ is a nut graph.
\end{proof}

This is a special case of the coalescence construction devised by Sciriha \cite{ScirihaCoalescence}. Corollary~21 
in~\cite{ScirihaCoalescence} is equivalent
to the statement that coalescence of any two nut graphs $G_1$ and $G_2$ at any pair of vertices $v_1 \in V(G_1)$ and 
$v_2 \in V(G_2)$ produces a nut graph. The fusion of a triangle and a pentagon is one of the three Sciriha graphs; see Figure~\ref{fig:orbit_count}(a). 
Note that the above construction is used on a single vertex of a nut graph $G$. Had we used it
iteratively on all vertices of $G$, that would give us yet another multiplier construction.
In fact, various sorts of mixed objects can be envisaged. After the initial application of $\mathcal{M}_3$ or $\mathcal{M}_5$ on an appropriate
parent graph, which gives rise to a nut graph, the way lies open to application of the coalescence construction, locally or globally. See Figure~\ref{fig:weirdMonster} for 
examples.

\begin{figure}[!htb]
\centering
\begin{subfigure}{0.5\textwidth}
\centering
    \includegraphics[scale=0.8]{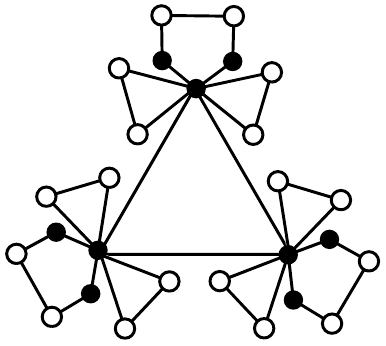}
    \caption{}
    \label{fig:maltMonster1}
\end{subfigure}%
\begin{subfigure}{0.5\textwidth}
\centering
    \includegraphics[scale=0.8]{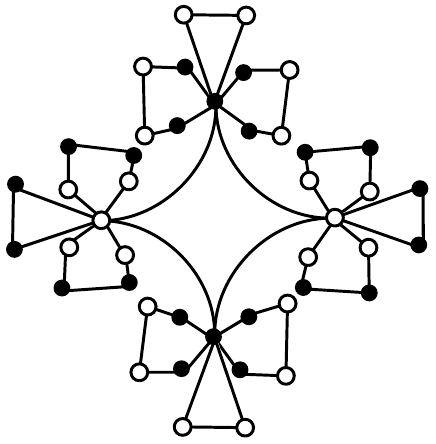}
    \caption{}
    \label{fig:maltMonster2}
\end{subfigure}%
\caption{These two nut graphs were obtained from (a) $C_3$ and (b) $C_4$ by an application of the $\mathcal{M}_3$ resp.\ $\mathcal{M}_5$, followed by repeated
application of Proposition~\ref{prop:specSciriha}. Entries in the kernel eigenvector in each graph are all of equal magnitude and represented by circles colour-coded for sign.}
\label{fig:weirdMonster}
\end{figure}

The triangle-multiplier and pentagon-multiplier constructions may be generalised to a $k$-multiplier construction:
Let $G$ be a $(2r)$-regular graph and let $k \geq 3$. Let $\mathcal{M}_k(G)$ be the graph obtained from $G$ by fusing a bouquet of $r$ $k$-cycles to every
vertex of $G$. In fact, Propositions~\ref{prop:multiplier3} and~\ref{prop:multiplier5} have natural generalisations to every $\mathcal{M}_k$, where \{$k \geq 3$ and $k \equiv 3 \pmod{4}$\}
and \{$k \geq 5$ and $k \equiv 1 \pmod{4}$\}, respectively. These generalisations follow immediately by the subdivision construction (see Section~\ref{sec:symmetrySec}).

\subsection{Characterisation of orders for nut graphs with \texorpdfstring{$\boldsymbol{2}$}{2} vertex orbits}
Observe that columns for prime values of $n$ are absent from Table~\ref{tbl:2}.
This is because the search did not reveal any examples in the range. As the next theorem shows,
this is no coincidence.

\begin{theorem}
\label{thm:numberTheoretic}
Let $G$ be a nut graph of order $n$ with precisely two vertex orbits. Then $n$ is not a prime number.
\end{theorem}

The next proposition will be useful in the proof of the above theorem.

\begin{proposition}
Let $G$ be a nut graph and let \,$\x = [x_1\ \ldots\ x_n]^\intercal \in \ker \mathbf{A}(G)$. 
If there exists a sign-reversing automorphism $\alpha \in \Aut(G)$ then
all orbits are of even size. Moreover, for every $j$, half of the entries $\{ x_i \mid i \in \mathcal{V}_j \}$ are 
positive, and the other half are negative.
\end{proposition}

\begin{proof}
Think of the automorphism $\alpha$ as a product of disjoint cycles. Note that elements of any given cycle of $\alpha$ are contained
in the same vertex orbit. Since $\alpha$ is a sign-reversing automorphism, every vertex
$i$ is mapped to a vertex $i^\alpha$ carrying the opposite sign (i.e.\ $x_i \cdot x_{i^\alpha} < 0$). Therefore, every
cycle of $\alpha$ is of even length and contains a perfect matching whose edges join vertices with entries
of opposite sign. Hence, $\{ x_i \mid i \in \mathcal{V}_j \}$ contains the same number of positive and negative elements.
\end{proof}

\begin{corollary}
\label{cor:numberTheoretic}
Let $G$ be a nut graph of order $n$. If $n$ is odd then $\Aut(G)$ does not contain any sign-reversing automorphism.
Moreover, kernel eigenvector entries are constant within a given orbit. 
\end{corollary}

\begin{proof}[Proof of Theorem~\ref{thm:numberTheoretic}]
Let $n_1 = |\mathcal{V}_1|$ and $n_2 = |\mathcal{V}_2|$ with $n = n_1 + n_2$ and $n_1, n_2 \geq 1$. Let $d_{ij}$ be the number of neighbours
of a vertex $v \in \mathcal{V}_i$ that reside in $\mathcal{V}_j$, where $1 \leq i, j \leq 2$. Since $G$ is simple and connected,
$0 \leq d_{ii} < n_i$ and $1 \leq d_{ji} \leq n_i$. The number of inter-orbit edges is
\begin{equation}
\label{eq:countEdges}
d_{12} n_1 = d_{21} n_2.
\end{equation}
The proof proceeds by contradiction. Suppose that $n$ is a prime.
The case $n = 2$ is trivial, since there are no nut graphs on $2$ vertices. As $n$ is odd, kernel eigenvector entries within each 
orbit are constant by Corollary~\ref{cor:numberTheoretic}. Let $a_i$ be the entry in $\mathcal{V}_i$.
The local conditions are
\begin{equation}
\label{eq:localTwoOrb}
\begin{aligned}
d_{11} a_1 + d_{12} a_2 & = 0, \\
d_{21} a_1 + d_{22} a_2 & = 0.
\end{aligned}
\end{equation}
First, we note that Equation~\ref{eq:countEdges} has a unique solution. Equation~\ref{eq:countEdges} implies
$$
d_{12} (n_1 + n_2) = (d_{21} + d_{12}) n_2.
$$
Since $n_1 + n_2$ is a prime factor, it divides either $d_{21} + d_{12}$ or $n_2$. As it clearly cannot divide $n_2$,
it divides $d_{21} + d_{12}$. But $d_{21} + d_{12} \leq n_1 + n_2$. Divisibility is possible only in the case where $d_{21} = n_1$
and $d_{12} = n_2$. For this case,
Equation~\eqref{eq:localTwoOrb} can be expressed in matrix form as
\begin{equation}
\begin{bmatrix}
n_1 &  d_{22}  \\
d_{11} & n_2  
\end{bmatrix}
\begin{bmatrix}
a_1  \\
a_2
\end{bmatrix} = 
\begin{bmatrix}
0  \\
0
\end{bmatrix}
\end{equation}
The determinant of the matrix is $n_1 n_2 - d_{11} d_{22} > 0$ and therefore $a_1 = a_2 = 0$.
This contradicts the fact that $G$ is a nut graph. Therefore, $n$ cannot be prime.
\end{proof}

We have justified the claim that nut graphs with two vertex orbits cannot be of prime order $n$.
We add to the picture by showing that a nut graph with two vertex orbits exists for all composite 
orders $n \geq 9$.

\begin{theorem}
\label{ref:compositeNutExistence}
Let $n \geq 9$ such that $n$ is not a prime. Then there exists a nut graph $G$ of order $n$ with 
$o_v(G) = 2$.
\end{theorem}

\begin{proof}
Let $n = p_1p_2 \cdots p_k$ be the decomposition of $n$ into prime factors.
Without loss of generality, we may assume that $p_1 \leq p_2 \leq \cdots \leq p_k$.

\vspace{0.5\baselineskip}
\noindent
\textbf{Case 1:} Suppose that $p_1 = 2$. If $n \not\equiv 0 \pmod{3}$ then Proposition~\ref{prop:roseWindow} guarantees a solution.
If $n \equiv 0 \pmod{3}$ then Proposition~\ref{prop:triangCyc} guarantees a solution. 

\vspace{0.5\baselineskip}
\noindent
\textbf{Case 2:} Now suppose that $p_1 > 2$. Clearly, $p_1$ is an odd integer. The strategy is to find a vertex transitive
graph $H$ of order $\widetilde{n} = n / p_1$ and degree $\widetilde{d} = p_1 - 1$. Then use the triangle-multiplier construction
to obtain $\mathcal{M}_3(H)$.
Since $n$ is not a prime $n / p_1 \geq p_1$ and so $\widetilde{n} > \widetilde{d}$. Note that $\widetilde{d}$ is an even number.
The circulant graph $H = \Circ(\widetilde{n}, \{1, 2, \ldots, \widetilde{d}/2\})$ has the prescribed order and degree and it is of course 
vertex transitive as required. The graph $\mathcal{M}_3(H)$ is of order $n$. By Proposition~\ref{prop:multiConstrSym}, this graph has two vertex orbits.
\end{proof}

Theorem~\ref{ref:compositeNutExistence} shows that there is at least one non-zero entry in every column of Table~\ref{tbl:2}. However, we believe that 
row $o_e = 3$ by itself consists of non-zero entries.

\begin{conjecture}
\label{conj:conjecturev2e3}
Let $n \geq 9$ such that $n$ is not a prime. Then there exists a nut graph $G$ of order $n$ with 
$o_v(G) = 2$ \emph{and} $o_e(G) = 3$. 
\end{conjecture}

Conjecture~\ref{conj:conjecturev2e3} holds for all even numbers (covered by Case 1 of the proof), all multiples of $3$ (covered by Proposition~\ref{prop:triangCyc}) and 
all perfect squares (in that case graph $H$ in the proof is a complete graph). 
The conjecture can also be validated for those values of $n$, such that there exists an edge transitive graph $H$
of order $\widetilde{n}$ and degree $\widetilde{d}$ in the proof of Theorem~\ref{ref:compositeNutExistence}.
Table~\ref{tbl:evilOrders} shows the orders up to $300$ that are not resolved by anything mentioned thus far.
For some of these orders we were able to provide graph $H$ (see proof of Theorem~\ref{ref:compositeNutExistence}).
N/A in the table indicates that no such graph $H$ exists (based on the census by Conder and Verret~\cite{Conder2019,marstonCensus}).
Order $35$, for example, cannot be resolved in this way, because
there is only one vertex transitive graph of order $7$ and degree $4$, namely $\Circ(7, \{1, 2\})$, but it is not edge transitive, and so a completely different approach is required.
For order $295$, $H$ would have to be a $4$-regular edge transitive graph of order $59$, but no such graph is known (see~\cite{Potocnik2020,Potocnik2015,Potocnik2013,Potocnik2009}).
The classes of edge transitive circulants are provided in~\cite{Onkey1996} could be used to resolve some orders beyond Table~\ref{tbl:evilOrders}.
All graphs $H$ provided in Table~\ref{tbl:evilOrders} are circulants. However, for some orders there are non-circulant alternative possibilities, e.g.\ the graph $C_5 \cart C_5$ 
 could be used for $n = 125$ and $\Cay(\mathbb{Z}_5 \times \mathbb{Z}_5, \{(0, 1), (1, 0), (1, 1)\})$ for order $n = 175$.

\begin{table}[!htb]
\centering
\begin{tabular}{ | r @{$\;=\;$} l | l |}
\hline
\multicolumn{2}{|c|}{Order} & Graph $H$ \\
\hline \hline
$35$ & $5 \cdot 7$ & N/A \\ 
$55$ & $5 \cdot 11$ & N/A \\
$65$ & $5 \cdot 13$ & $\Circ(13, \{1, 5\})$ \\
$77$ & $7 \cdot 11$ & N/A \\
$85$ & $5 \cdot 17$ & $\Circ(17, \{1, 4\})$ \\
$91$ & $7 \cdot 13$ & $\Circ(13, \{1, 3, 4\})$ \\
$95$ & $5 \cdot 19$ & N/A \\
$115$ & $5 \cdot 23$ & N/A \\
$119$ & $7 \cdot 17$ & N/A \\
$125$ & $5^3$ & $\Circ(25, \{1, 7]\}$ \\ 
$133$ & $7 \cdot 19$ & $\Circ(19, \{1, 7, 8\})$ \\
$143$ & $11 \cdot 13$ & N/A \\
$145$ & $5 \cdot 29$ & $\Circ(29, \{1, 12\})$ \\
$155$ & $5 \cdot 31$ & N/A \\
$161$ & $7 \cdot 23$ & N/A \\
$175$ & $5^2 \cdot 7$ & $\Circ(35, \{1, 6\})$ \\ 
$185$ & $5 * 37$ & Circ(37, [1, 6]) \\
\hline
\end{tabular}
\qquad
\begin{tabular}{ | r @{$\;=\;$} l | l |}
\hline
\multicolumn{2}{|c|}{Order} & Graph $H$ \\
\hline \hline
$187$ & $11 \cdot 17$ & N/A \\
$203$ & $7 \cdot 29$ & N/A \\
$205$ & $5 \cdot 41$ & $\Circ((41, \{1, 9\})$ \\
$209$ & $11 \cdot 19$ & N/A \\
$215$ & $5 \cdot 43$ & N/A \\
$217$ & $7 \cdot 31$ & $\Circ(31, \{1, 5, 6\})$ \\
$221$ & $13 \cdot 17$ & N/A \\
$235$ & $5 \cdot 47$ & N/A \\
$245$ & $5 \cdot 7^2$ & $\Circ(35, \{1, 11, 16\})$ \\
$247$ & $13 \cdot 19$ & N/A \\
$253$ & $11 \cdot 23$ & N/A \\
$259$ & $7 \cdot 37$ & $\Circ(37, \{1, 10, 11\})$ \\
$265$ & $5 \cdot 53$ & $\Circ(53, \{1, 23\})$  \\ 
$275$ & $5^2 \cdot 11$ & $\Circ(25, \{1, 4, 6, 9, 11\})$ \\ 
$287$ & $7 \cdot 41$ & N/A \\ 
$295$ & $5 \cdot 59$ & Unknown \\ 
$299$ & $13 \cdot 23$ & N/A \\ 
\hline
\end{tabular}
\caption{List of all integers $9 \leq n \leq 300$ that are not prime, not even, not multiples of three and not perfect squares.
Where a graph $H$ is listed, it proves Conjecture~\ref{conj:conjecturev2e3} for the particular order. N/A indicates that no graph
$H$ with the desired properties exists. For order $295$ it is not known whether such a graph exists.}
\label{tbl:evilOrders}
\end{table}

\begin{question}
\label{quest:quest24}
Find a nut graph $G$ with $2$ vertex orbits and $3$ edge orbits for orders $n = 35, 55, 77, 95, \ldots$ and other non-resolved orders.
\end{question}

\section{How constructions influence symmetry}
\label{sec:symmetrySec}

Several constructions have been described for producing a larger nut graph when applied to a smaller nut graph $G$;
literature examples include the bridge construction (insertion of two vertices on a bridge) \cite{ScirihaGutman-NutExt}, the subdivision construction (insertion of four vertices
on an edge) \cite{ScirihaGutman-NutExt} and the so-called Fowler construction \cite{GPS}, which has the net result of introducing $2d$ new vertices in the
proximity of a vertex $v$ of degree $d$. In Section~\ref{sec:multiConstructions} we have given examples of constructions that do not require 
the parent $G$ to be a nut graph. Here, we are interested in the implications of the various constructions for numbers of vertex and edge orbits of
the constructed nut graph.

\begin{proposition}
\label{prop:bridgeStuff}
Let $G$ be a nut graph and let $e = uv \in E(G)$ be a bridge in $G$. Let $\mathcal{E}$ be the orbit of the bridge $e$ under $\Aut(G)$.
The graph obtained from $G$ by applying the bridge construction on every edge from $\mathcal{E}$, denoted  $B(G, \mathcal{E})$, 
is a nut graph and $\Aut(G) \leq \Aut(B(G, \mathcal{E}))$. 

If, in addition, $\Aut(G) \cong \Aut(B(G, \mathcal{E}))$ then the following statements hold.
\begin{enumerate}[label=(\roman*)]
\item
If there exists an element $\varphi \in \Aut(G)$ such that $u^\varphi = v$ and $v^\varphi = u$, then
$o_v(B(G, \mathcal{E})) = o_v(G) + 1$ and $o_e(B(G, \mathcal{E})) = o_e(G) + 1$.
\item
If there is no element $\varphi \in \Aut(G)$ such that $u^\varphi = v$ and $v^\varphi = u$, then
$o_v(B(G, \mathcal{E})) = o_v(G) + 2$ and $o_e(B(G, \mathcal{E})) = o_e(G) + 2$.
\end{enumerate}
\end{proposition}

\begin{proof}
It is clear that $B(G, \mathcal{E})$ is a nut graph \cite{ScirihaGutman-NutExt}. Every element $\alpha \in \Aut(G)$ can be extended in a 
natural way to an element $\widehat{\alpha} \in \Aut(B(G, \mathcal{E}))$. More precisely, since the arc $(u, v)$ was subdivided, so was
its image $(u^\alpha, v^\alpha)$. Let the new vertices on $(u, v)$ be labeled $w_1$ and $w_2$, where $w_1$ is adjacent
to $u$. And let the new vertices on $(u^\alpha, v^\alpha)$ be labeled $w'_1$ and $w'_2$ where $w'_1$ is
adjacent to $u^\alpha$. Then $w_1^{\widehat{\alpha}} = w'_1$ and $w_2^{\widehat{\alpha}} = w'_2$.
This immediately implies that $\Aut(G) \leq \Aut(B(G, \mathcal{E}))$.

Note that $\Aut(B(G, \mathcal{E}))$ may include additional automorphisms that were not induced by $\Aut(G)$. These may
cause merging of vertex orbits and merging of edge orbits. If $\Aut(G) \cong \Aut(B(G, \mathcal{E}))$, then we know
that there are no such additional automorphisms. Note that graph $B(G, \mathcal{E})$ has at most two new vertex orbits, namely,
the orbit of $w_1$ and the orbit of $w_2$. The edge $uv$ was substituted by the three edges $uw_1, w_1w_2$ and $w_2v$.
If there exists an element $\varphi \in \Aut(G)$ such that $u^\varphi = v$ and $v^\varphi = u$, then using its extension 
$\widetilde{\varphi}$ we get $w_1^{\widehat{\varphi}} = w_2$. This means that vertices 
$w_1$ and $w_2$ are in the same
vertex orbit. Similarly, edges $uw_1$ and $w_2v$ are in the same edge orbit. The claim follows.
\end{proof}

Note that since $\Aut(G) \leq \Aut(B(G, \mathcal{E}))$, the condition $|\!\Aut(G)| = |\!\Aut(B(G, \mathcal{E}))|$ automatically
implies $\Aut(G) \cong \Aut(B(G, \mathcal{E}))$.

The proof of the following proposition is analogous to that of Proposition~\ref{prop:bridgeStuff} and is skipped here.

\begin{proposition}
\label{prop:edgeStuff}
Let $G$ be a nut graph and let $e = uv \in E(G)$ be an edge in $G$. Let $\mathcal{E}$ be the orbit of the edge $e$ under $\Aut(G)$.
The graph obtained from $G$ by applying the subdivision construction on every edge from $\mathcal{E}$, denoted  $S(G, \mathcal{E})$, 
is a nut graph and $\Aut(G) \leq \Aut(S(G, \mathcal{E}))$. 

If, in addition,  $\Aut(G) \cong \Aut(S(G, \mathcal{E}))$ then the following statements hold.
\begin{enumerate}[label=(\roman*)]
\item
If there exists an element $\varphi \in \Aut(G)$ such that $u^\varphi = v$ and $v^\varphi = u$, then
$o_v(S(G, \mathcal{E})) = o_v(G) + 2$ and $o_e(S(G, \mathcal{E})) = o_e(G) + 2$.
\item
If there is no element $\varphi \in \Aut(G)$ such that $u^\varphi = v$ and $v^\varphi = u$, then
$o_v(S(G, \mathcal{E})) = o_v(G) + 4$ and $o_e(S(G, \mathcal{E})) = o_e(G) + 4$.
\end{enumerate}
\end{proposition}

\begin{definition}
\label{def:fowler}
Let $G$ be a nut graph and let $v \in V(G)$ be a vertex of degree $d$ in $G$. Let $N(v) = \{u_1, \ldots, u_d\}$.
The graph $F(G, v)$, is obtained from $G$ in the following way: (a) edges incident to $v$ are deleted;
(b) let $w_1, \ldots, w_d$ and $x_1, \ldots, x_d$ denote $2d$ newly added vertices; (c) new edges
are added such that $x_i \sim u_i$ for $i = 1, \ldots, d$; and $x_i \sim w_j$ for $i \neq j$, $1 \leq i, j \leq d$; and $w_i \sim v$ for all $i = 1, \ldots, d$. The construction $F(G, v)$ has been called `the Fowler construction' 
in the nut-graph literature \cite{GPS}.
\end{definition}

Figure~\ref{fig:fowler_constr} illustrates the definition. Note that $u_1, \ldots, u_d$ are at distance $3$ from $v$ in $F(G, v)$. Moreover, 
degrees of all newly added vertices are $d$.

\begin{proposition}
\label{prop:technicalStuff}
Let $G$ be a nut graph and let $v \in V(G)$ be a vertex in $G$. Let $\mathcal{V}$ be the orbit of the vertex $v$ under $\Aut(G)$.
The graph obtained from $G$ by applying the Fowler construction on every vertex from $\mathcal{V}$, denoted  $F(G, \mathcal{V})$, 
is a nut graph and $\Aut(G) \leq \Aut(F(G, \mathcal{V}))$.

Suppose that, in addition, $\Aut(G) \cong \Aut(F(G, \mathcal{V}))$. Let the vertices in the neighbourhood of $v$ in $G$ and in the first, second and third neighbourhood of $v$ in $F(G, \mathcal{V})$ be labeled as in Definition~\ref{def:fowler}.   The stabiliser $\Aut(G)_v$ fixes $N(v)$ set-wise in the
graph $G$ and partitions $N(v)$ into $t$ orbits. Let $\mathcal{S} = \{(w_i, x_j) \mid i \neq j; 1 \leq i, j \leq d \}$.  $\Aut(F(G, \mathcal{V}))_v$ partitions $\mathcal{S}$ into $\tau$ orbits. Then $o_v(F(G, \mathcal{V})) = o_v(G) + 2t$ and $o_e(F(G, \mathcal{V})) = o_e(G) + t + \tau$.
\end{proposition}

Note that we define the action of $\Aut(F(G, \mathcal{V}))_v$ on pairs $(w_i, x_j)$ by taking $(w_i, x_j)^\alpha = (w_i^\alpha, x_j^\alpha)$,
where $\alpha \in \Aut(F(G, \mathcal{V}))_v$. The proof of the above proposition uses the same approach as that of
 Proposition~\ref{prop:bridgeStuff} and is skipped here.

\begin{figure}[!htbp]
\centering
\subcaptionbox{$G$\label{Fig-FowlerExta}}[.5\linewidth]{
\centering
\begin{tikzpicture}
\definecolor{mygreen}{RGB}{19, 56, 190}
\definecolor{sapphire}{RGB}{82, 178, 191}
\tikzstyle{vertex}=[draw,circle,font=\scriptsize,minimum size=7pt,inner sep=1pt,fill=mygreen]
\tikzstyle{edge}=[draw,thick]
\coordinate (u1) at (-2, -1);
\coordinate (u2) at (-0.7, -1);
\coordinate (ud) at (2, -1);
\path[edge,fill=sapphire!60!white] (u1) .. controls ($ (u1) + (-60:0.5) $) and ($ (u2) + (-120:0.5) $) .. (u2) .. controls ($ (u2) + (-60:0.7) $) and ($ (ud) + (-120:0.7) $)
.. (ud) .. controls ($ (ud) + (-60:3) $) and ($ (u1) + (-120:3) $) .. (u1);
\node[vertex,fill=mygreen,label={[yshift=0pt,xshift=0pt]90:$v$}] (v) at (0, 0) {};
\node[vertex,fill=mygreen,label={[xshift=2pt,yshift=0pt]180:$u_1$}] (ux1) at (u1) {};
\node[vertex,fill=mygreen,label={[xshift=2pt,yshift=0pt]180:$u_2$}] (ux2) at (u2) {};
\node[vertex,fill=mygreen,label={[xshift=-2pt,yshift=0pt]0:$u_d$}] (uxd) at (ud) {};
\node at (0.5, -1) {$\ldots$};
\path[edge] (v) -- (ux1) node[midway,xshift=-23pt,yshift=0,color=black] {};
\path[edge] (v) -- (ux2) node[midway,xshift=16pt,yshift=0,color=black] {};
\path[edge] (v) -- (uxd) node[midway,xshift=20pt,yshift=0,color=black] {};
\end{tikzpicture}
}%
\subcaptionbox{$F(G, v)$\label{Fig-FowlerExtb}}[.5\linewidth]{
\centering
\begin{tikzpicture}
\definecolor{mygreen}{RGB}{19, 56, 190}
\definecolor{sapphire}{RGB}{82, 178, 191}
\tikzstyle{vertex}=[draw,circle,font=\scriptsize,minimum size=7pt,inner sep=1pt,fill=mygreen]
\tikzstyle{edge}=[draw,thick]
\coordinate (u1) at (-2, -4);
\coordinate (u2) at (-0.7, -4);
\coordinate (ud) at (2, -4);
\path[edge,fill=sapphire!60!white] (u1) .. controls ($ (u1) + (-60:0.5) $) and ($ (u2) + (-120:0.5) $) .. (u2) .. controls ($ (u2) + (-60:0.7) $) and ($ (ud) + (-120:0.7) $)
.. (ud) .. controls ($ (ud) + (-60:3) $) and ($ (u1) + (-120:3) $) .. (u1);
\node[vertex,fill=mygreen,label={[yshift=2pt,xshift=0pt]90:$v$}] (v) at (0, -1) {};
\node[vertex,fill=mygreen,label={[xshift=2pt,yshift=0pt]180:$w_1$}] (q1) at (-2, -2) {};
\node[vertex,fill=mygreen,label={[xshift=-2pt,yshift=0pt]0:$w_2$}] (q2) at (-0.7, -2) {};
\node[vertex,fill=mygreen,label={[xshift=-2pt,yshift=0pt]0:$w_d$}] (qd) at (2, -2) {};
\node at (1, -2) {$\ldots$};
\node[vertex,fill=mygreen,label={[xshift=2pt,yshift=-2pt]180:$x_1$}] (p1) at (-2, -3) {};
\node[vertex,fill=mygreen,label={[xshift=2pt,yshift=-2pt]0:$x_2$}] (p2) at (-0.7, -3) {};
\node[vertex,fill=mygreen,label={[xshift=0pt,yshift=-2pt]0:$x_d$}] (pd) at (2, -3) {};
\node at (1, -3) {$\ldots$};
\node[vertex,fill=mygreen,label={[xshift=2pt,yshift=0pt]180:$u_1$}] (ux1) at (-2, -4) {};
\node[vertex,fill=mygreen,label={[xshift=2pt,yshift=0pt]0:$u_2$}] (ux2) at (-0.7, -4) {};
\node[vertex,fill=mygreen,label={[xshift=-2pt,yshift=0pt]0:$u_d$}] (uxd) at (2, -4) {};
\node at (1, -4) {$\ldots$};
\path[edge] (v) -- (q1);
\path[edge] (v) -- (q2);
\path[edge] (v) -- (qd);
\path[edge] (p1) -- (q2); \path[edge] (p1) -- (qd);
\path[edge] (p2) -- (q1); \path[edge] (p2) -- (qd);
\path[edge] (pd) -- (q1); \path[edge] (pd) -- (q2);
\path[edge] (p1) -- (ux1) node[midway,xshift=-15pt,yshift=0,color=black] {}; 
\path[edge] (p2) -- (ux2) node[midway,xshift=16pt,yshift=0,color=black] {}; 
\path[edge] (pd) -- (uxd) node[midway,xshift=16pt,yshift=0,color=black] {};
\end{tikzpicture}
}
\caption{A construction for expansion of a nut graph $G$ about vertex $v$ of degree $d$, to give $F(G, v)$.
Panel (a) shows the neighbourhood of vertex $v$ in $G$. 
Panel (b) shows additional vertices and edges in $F(G, v)$.}
\label{fig:fowler_constr}
\end{figure}
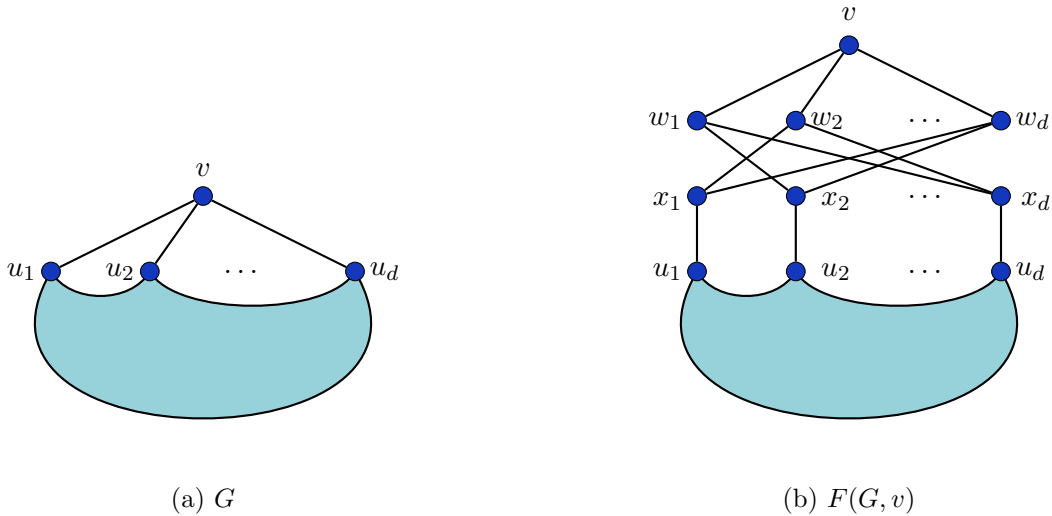

We are interested in the growth of the number of edge orbits under the construction.
Let us define $\Phi(G, v) = o_e(F(G, \mathcal{V})) - o_e(G)$, where $\mathcal{V}$ is the orbit of the vertex $v$.
Proposition~\ref{prop:technicalStuff} has the following corollary.

\begin{corollary}
\label{cor:growthBounds}
Let $G$ be a nut graph and let $v \in V(G)$ be a vertex in $G$. Let $\mathcal{V}$ be the orbit of the vertex $v$ under $\Aut(G)$.
Let $F(G, \mathcal{V})$ be as in Proposition~\ref{prop:technicalStuff} and also let $\Aut(G) \cong \Aut(F(G, \mathcal{V}))$. Then
\begin{equation}
\label{eq:orbitExcess}
4 \leq \Phi(G, v) \leq d^2.
\end{equation}
If $\deg(v) \geq 3$, then $\Phi(G, v) \geq 5$.
\end{corollary}

\begin{proof}
To get the upper bound, assume that each vertex of $N(v)$ is in its own orbit and therefore $t = d$. 
Similarly, each element of $\mathcal{S}$ is in its own orbit and therefore $\tau = d^2 - d$.

For parameters $t$ and $\tau$ from Proposition~\ref{prop:technicalStuff}, it holds that $t \geq 1$ and $\tau \geq 1$. Therefore,
$2 \leq o_e(F(G, \mathcal{V})) - o_e(G)$. By Lemma~\ref{lem:7gen}, $\alpha \in \Aut(G)_v$ cannot be sign-reversing. There must
be at least one vertex in $N(v)$ with a positive entry in the kernel eigenvector and at least one with a negative entry. Therefore,
the group $\Aut(G)_v$ partitions $N(v)$ in at least $2$ orbits, say $\mathcal{U}$ and $\mathcal{U}'$. Vertices $x_1, \ldots, x_d$ cannot be in the same orbit under $\Aut(G)_v$ as any
other vertex of $F(G, v)$; see Figure~\ref{fig:fowler_constr}. The same is true for $w_1, \ldots, w_d$. Since $\Aut(G) \cong \Aut(F(G, \mathcal{V}))$, $\{x_1, \ldots, x_d\}$ and $\{w_1, \ldots, w_d\}$ are partitioned into orbits under $\Aut(G)_v$ in the same way as $\{u_1, \ldots, u_d\}$ (i.e.,
$x_i$ and $x_j$ belong to the same orbit if and only if $u_i$ and $u_j$ belong to the same orbit). Orbits $\mathcal{U}$ and $\mathcal{U}'$
induce orbits $\mathcal{X}$ and $\mathcal{X}'$ on $\{x_1, \ldots, x_d\}$ and orbits $\mathcal{W}$ and $\mathcal{W}'$ on $\{w_1, \ldots, w_d\}$. There exists at least one edge connecting $\mathcal{U}$ to $\mathcal{X}$, at least one edge connecting $\mathcal{U}'$ to $\mathcal{X}'$, at least one edge connecting $\mathcal{X}$ to $\mathcal{W}'$ and at least one edge connecting $\mathcal{X}'$ to $\mathcal{W}$. Each of these four edges must be in a distinct new edge orbit, hence $4 \leq o_e(F(G, \mathcal{V})) - o_e(G)$.

If $\deg(v) \geq 3$ then either $\mathcal{U}$ or $\mathcal{U}'$ contains at least $2$ vertices. Without loss of generality
assume that $|\mathcal{U}| \geq 2$. Then there exists at least one edge connecting $\mathcal{X}$ to $\mathcal{W}$. This edge cannot
share the orbit with any of the above four edges, hence $5 \leq o_e(F(G, \mathcal{V})) - o_e(G)$.
\end{proof}

The upper bound is best possible, because the equality in \eqref{eq:orbitExcess} can be attained if we take $G$ to be any asymmetric
nut graph (i.e.\ $|\Aut(G)| = 1$). This bound is attained even within the class of vertex transitive graphs, when we take $G$ to be a
GRR~\cite{WatkinsGRR1,WatkinsGRR2} nut graph, such as the one in Figure~\ref{fig:zero_sym_nut}. 

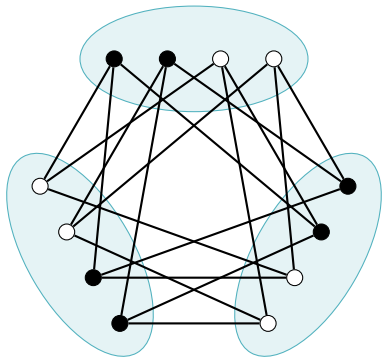
\begin{figure}[!htbp]
\centering
\begin{tikzpicture}
\definecolor{mygreen}{RGB}{19, 56, 190}
\definecolor{sapphire}{RGB}{82, 178, 191}
\draw[rotate=60,,color=sapphire,fill=sapphire!15!white] (0, 0) circle (1.5cm and 0.7cm);
\draw[rotate around={-60:(-3, 0)},color=sapphire,fill=sapphire!15!white] (-3, 0) circle (1.5cm and 0.7cm);
\draw[rotate around={0:(120:3)},color=sapphire,fill=sapphire!15!white] (120:3) circle (1.5cm and 0.7cm);
\tikzstyle{vertex}=[draw,circle,font=\scriptsize,minimum size=6pt,inner sep=1pt,fill=mygreen]
\tikzstyle{edge}=[draw,thick]
\foreach \x/\cc/\fc in {0/black/white,1/black/white,2/white/black,3/white/black} {
     \node[vertex,fill=\cc] (u\x) at ($ (120:3) + (0.7 * \x, 0) - (1.05, 0)  $) {};
     \node[vertex,fill=\fc] (v\x) at ($ (0:0) + ({0.7 * cos(60) * \x}, {0.7 * sin(60) * \x}) - (60:1.05)  $) {};
     \node[vertex,fill=\cc] (w\x) at ($ (-3,0) + ({0.7 * cos(120) * \x}, {0.7 * sin(120) * \x}) - (120:1.05)  $) {};
}
   \path[edge] (v0) -- (u2);
   \path[edge] (v1) -- (u3);
   \path[edge] (v2) -- (u0);
   \path[edge] (v2) -- (u2);
   \path[edge] (v3) -- (u1);
   \path[edge] (v3) -- (u3);
  \path[edge] (u3) -- (w2);
   \path[edge] (u2) -- (w3);
   \path[edge] (u1) -- (w0);
   \path[edge] (u1) -- (w2);
   \path[edge] (u0) -- (w1);
   \path[edge] (u0) -- (w3);
   \path[edge] (v3) -- (w1);
   \path[edge] (v2) -- (w0);
   \path[edge] (v1) -- (w3);
   \path[edge] (v1) -- (w1);
   \path[edge] (v0) -- (w2);
   \path[edge] (v0) -- (w0);
\end{tikzpicture}
\caption{The smallest GRR nut graph has order $12$ and degree $6$. The graph contains three cliques represented
by shaded regions; edges within cliques are not drawn.
Entries in the kernel eigenvector are all of equal magnitude and represented by circles colour-coded for sign.}
\label{fig:zero_sym_nut}
\end{figure}

The lower bound is more interesting. The restriction of $G$ to nut graphs in Corollary~\ref{cor:growthBounds} is significant since, for example, if we take a complete graph on $n \geq 4$ vertices then $o_e(F(K_n, \mathcal{V})) - o_e(K_n) = 2$. For vertices of degree $2$, Corollary~\ref{cor:growthBounds} implies $o_e(F(K_n, \mathcal{V})) - o_e(K_n) = 4$. Small graphs with
vertices of degree $d \in \{3, 4\}$  furnish examples where $o_e(F(K_n, \mathcal{V})) - o_e(K_n) = 5$; see Figure~\ref{fig:smallExamplesGrowth}.
\begin{figure}[!htbp]
\centering
\subcaptionbox{}{
\begin{tikzpicture}[scale=0.4]
\tikzstyle{vertex}=[draw,circle,font=\scriptsize,minimum size=6pt,inner sep=1pt,fill=black]
\tikzstyle{edge}=[draw,thick]
\node[vertex] (0) at (1, 4) {};
\node[vertex] (5) at (-5, 4) {};
\node[vertex] (7) at (-2, 1.5) {};
\node[vertex] (1) at (-3.5, 2.5) {};
\node[vertex] (3) at (-0.5, 2.5) {};
\node[vertex,fill=magenta] (2) at (-2, -0.5) {};
\node[vertex] (4) at (-5, -2.5) {};
\node[vertex] (6) at (1, -2.5) {};
 \path[edge]
(0) edge node {} (3) 
(0) edge node {} (5) 
(0) edge node {} (6) 
(1) edge node {} (4) 
(1) edge node {} (5) 
(1) edge node {} (7) 
(2) edge node {} (4) 
(2) edge node {} (6) 
(2) edge node {} (7) 
(3) edge node {} (6) 
(3) edge node {} (7) 
(4) edge node {} (5) 
(4) edge node {} (6) 
(4) edge node {} (7) 
(6) edge node {} (7) ;
\end{tikzpicture}
}
\qquad\qquad
\subcaptionbox{}{
\begin{tikzpicture}[scale=0.8]
\tikzstyle{vertex}=[draw,circle,font=\scriptsize,minimum size=6pt,inner sep=1pt,fill=black]
\tikzstyle{edge}=[draw,thick]
\node[vertex] (3) at (-2.2, 2.5) {};
\node[vertex] (0) at (-2.2, 3.5) {};
\node[vertex] (1) at (-2.2, 1.5) {};
\node[vertex] (6) at (0.4, 3.5) {};
\node[vertex] (4) at (-0.2, 2.5) {};
\node[vertex] (8) at (0.4, 1.5) {};
\node[vertex,fill=magenta] (2) at (2.7, 4) {};
\node[vertex,fill=magenta] (5) at (-4.8, 3.6) {};
\node[vertex] (7) at (-4.2, 4.7) {};
 \path[edge]
(0) edge node {} (4) 
(0) edge node {} (5) 
(0) edge node {} (6) 
(0) edge node {} (7) 
(1) edge node {} (4) 
(1) edge node {} (5) 
(1) edge node {} (7) 
(1) edge node {} (8) 
(2) edge node {} (4) 
(2) edge node {} (6) 
(2) edge node {} (7) 
(2) edge node {} (8) 
(3) edge node {} (5) 
(3) edge node {} (6) 
(3) edge node {} (7) 
(3) edge node {} (8) 
(4) edge node {} (6) 
(4) edge node {} (8) 
(5) edge node {} (7) 
(6) edge node {} (8);
\end{tikzpicture}
}
\caption{The smallest graphs where $\Phi(G, v) = 5$. 
 (a) For $d = 3$ the smallest example is unique and of order $8$. 
(b) For $d = 4$ one of the two smallest examples of order $9$ is shown. Vertices for which the bound is met are coloured magenta.}
\label{fig:smallExamplesGrowth}
\end{figure}
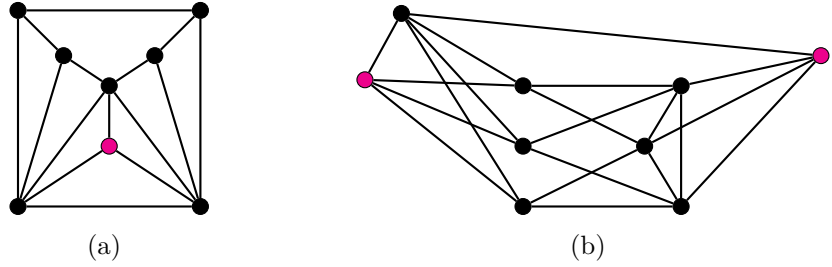
The search on nut graphs of orders $n \leq 12$ yields examples for $d \in \{5, 6, 7\}$ for which $\Phi(G, v) = 6$;
see \cite{GitHub}. 
Examples of graph with small $\Phi(G, v)$ can also be found in the class of regular graphs; see \cite{GitHub}. 
Moreover, examples with small $\Phi(G, v)$ can be found in the class of vertex transitive graphs. For example, $K_{3,3} \cart K_4$ is a sextic vertex transitive
nut graph with $\Phi(K_{3,3} \cart K_4, v) = 6$.

It is important to note that in Propositions~\ref{prop:bridgeStuff}, \ref{prop:edgeStuff} and \ref{prop:technicalStuff} the respective
construction was applied to all edges/vertices within a given edge/vertex orbit. This ensures that the graph obtained by
the construction inherits all the symmetries of the original graph $G$. 
The requirement that
 the order of the automorphism group of the graph does not increase upon applying the given construction, i.e.\ 
$\Aut(G) \cong \Aut(F(G, \mathcal{V}))$,  is also crucial to the propositions. Figure~\ref{fig:orbit_constr1} illustrates some of the complications
that can arise.

\begin{figure}[!htbp]
\centering
\subcaptionbox{$\Omega = (12, 13, 8)$\label{subfig:orbit_constr1a}}
{
\includegraphics[scale=0.6]{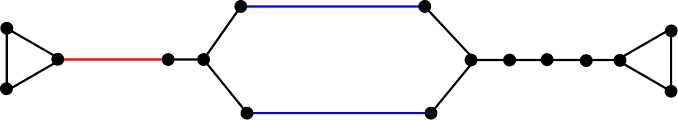}
}
\quad
\subcaptionbox{$\Omega = (16, 17, 8)$\label{subfig:orbit_constr1b}}
{
\includegraphics[scale=0.6]{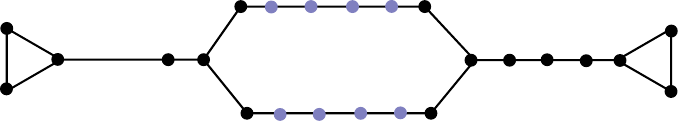}
}
\subcaptionbox{$\Omega = (7, 8, 16)$\label{subfig:orbit_constr1c}}
{
\includegraphics[scale=0.6]{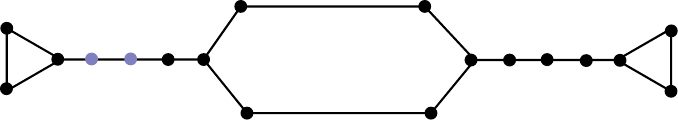}
}
\quad
\subcaptionbox{$\Omega = (18, 20, 4)$\label{subfig:orbit_constr1d}}
{
\includegraphics[scale=0.6]{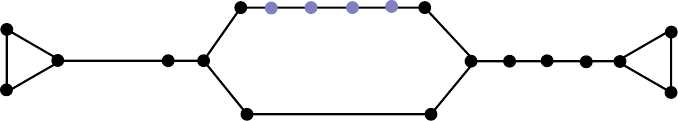}
}

\vspace{\baselineskip}
\subcaptionbox{$\Omega = (3, 4, 6)$\label{subfig:orbit_constr2itself}}
{
 \includegraphics[scale=0.6]{img/Fig1c.pdf}
\quad
}
\quad
\subcaptionbox{$\Omega = (7, 10, 2)$\label{subfig:orbit_constr2a}}
{
\quad
\includegraphics[scale=0.6]{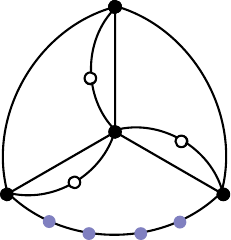}
\quad
}
\quad
\subcaptionbox{$\Omega = (9, 12, 2)$\label{subfig:orbit_constr2b}}
{
\quad
\includegraphics[scale=0.6]{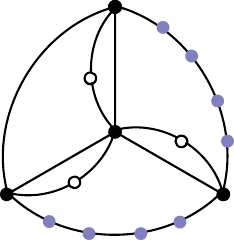}
\quad
}
\quad
\subcaptionbox{$\Omega = (5, 6, 6)$ \label{subfig:orbit_constr2c}}
{
\quad
\includegraphics[scale=0.6]{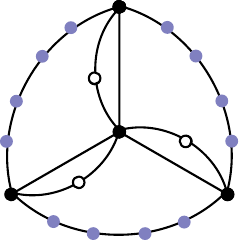}
}

\vspace{\baselineskip}
\subcaptionbox{$\Omega = (2, 4, 288)$ \label{subfig:orbit_constr3a}}
{
\includegraphics[scale=0.6]{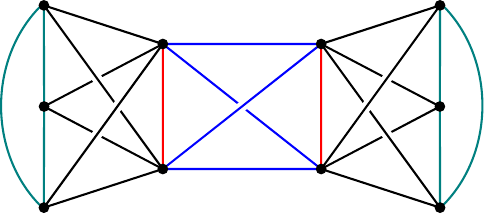}
}
\subcaptionbox{$\Omega = (12, 17, 16)$ \label{subfig:orbit_constr3c}}
{
\includegraphics[scale=0.6]{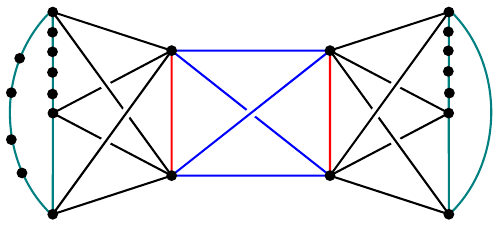}
}
\subcaptionbox{$\Omega = (7, 10, 32)$ \label{subfig:orbit_constr3b}}
{
\includegraphics[scale=0.6]{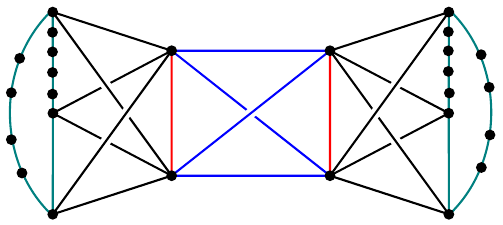}
}
\caption{Interplay of constructions, orbits and automorphism group. 
Parts (a) to (d): in the starting graph (a) two orbits are marked in red and blue; in (b)
subdivision on an entire orbit preserves the automorphism group; in (c) the bridge construction on an entire orbit leads to doubling in order of the automorphism
group; in (d) subdivision in a part orbit leads to halving. Parts (e) to (h): the Sciriha graph $S_3$ shown in (e) is progressively subdivided;
in (f) subdivision in a part orbit gives broken symmetry; in (g) subdivision in a part orbit of (f) preserves the order of the automorphism group; in (h) subdivision
of the final edge causes merging of orbits and restoration of the original symmetry. Parts (i) to (k): the highly symmetric graph (i) has four edge orbits; in (j) the symmetry is
considerably reduced by applying subdivision in a part of the green orbit of (i); the graph in (k) is obtained from (j) by applying subdivision on one of two equivalent edges
but the symmetry increases; subdivision of the two remaining green edges would restore the full symmetry of (i).
The signature $\Omega$ given for each graph $G$ denotes the triple $(o_v(G), o_e(G), |\mathrm{Aut}(G)|)$ as in Figure~\ref{fig:orbit_count}.
}
\label{fig:orbit_constr1}
\end{figure}

\begin{proposition}
\label{prop:multiConstrSym}
Let $G$ be a connected $(2t)$-regular graph, where $t \geq 1$. 
Then $o_v(\mathcal{M}_3(G)) = 2 o_v(G)$ and $o_e(\mathcal{M}_3(G)) = o_e(G) + 2o_v(G)$.
Moreover, $|\Aut(\mathcal{M}_3(G))| = (2^t t!)^{|V(G)|} |\Aut(G)|$.
\end{proposition}

\begin{proof}
Every element $\alpha \in \Aut(G)$ can be extended in a natural way to an element $\widehat{\alpha} \in \mathcal{M}_3(G)$; the element $\widehat{\alpha}$
moves the vertices of the original graph in the same way as $\alpha$, and also moves the corresponding attached triangles. Therefore, $|\Aut(\mathcal{M}_3(G))| \geq|\Aut(G)|$.
Now, let us consider action of the stabiliser within $\Aut(\mathcal{M}_3(G))$ that fixes the subgraph $G$.
Consider the triangles attached to an arbitrary vertex $v \in V(G)$. Clearly, the stabiliser permutes the $t$ triangles; this contributes $t!$ to the order of the stabiliser. In addition, there exist
involutions that swap two degree-$2$ endvertices in any attached triangle; this contributes $2^t$ to the order of the stabiliser. Finally, all these operations can be done independently
at every vertex $v \in V(G)$. Therefore, the order of the stabiliser is $ (2^t t!)^{|V(G)|}$. Using the Orbit-Stabiliser Lemma~\cite[Lemma~2.2.2]{Godsil2001}, we obtain $|\Aut(\mathcal{M}_3(G))| = (2^t t!)^{|V(G)|} |\Aut(G)|$. Now that the full automorphism group of $\mathcal{M}_3(G))$ is known, counting the vertex- and edge-orbits is straightforward.
\end{proof}

\begin{example}
\label{examp:2}
Consider graphs $K_7$, $\Circ(12, \{1, 5\})$ and the hypercube $Q_6$. All these graphs are vertex and edge transitive.
Let us determine the order of the automorphism group of $\mathcal{M}_3(G)$ for $G$ from the above list.
It is easy to see that $|\Aut(\Circ(12, \{1, 5\}))| = 2^8 \cdot 3$, $|\Aut(K_7)| = 7!$ and $|\Aut(Q_6)| = 2^6 \cdot 6!$.
Graphs $K_7$ and $Q_6$ are $6$-regular, while $\Circ(12, \{1, 5\})$ is $4$-regular. By Proposition~\ref{prop:multiConstrSym},
\begin{align*}
|\Aut(\mathcal{M}_3(\Circ(12, \{1, 5\}))| & = (2^2 \cdot 2!)^{12} \cdot (2^8 \cdot 3) = 52776558133248, \\
|\Aut(\mathcal{M}_3(K_7))| & = (2^3 \cdot 3!)^7 \cdot 7! = 2958824445050880, \\
|\Aut(\mathcal{M}_3(Q_6))| & = (2^3 \cdot 3!)^{64} \cdot (2^6 \cdot 6!)  \approx 1.832 \cdot 10^{112}.
\end{align*}
\end{example}

Note that even though the automorphism group of $\mathcal{M}_3(G)$ might be absurdly large, the numbers of vertex and
edge orbits remain small, and in determining them we can ignore the extra symmetries.~$\diamond$

\vspace{0.5\baselineskip}
\noindent
Proposition~\ref{prop:multiConstrSym} has a natural generalisation.
\begin{proposition}
\label{prop:multiConstrSymGen}
Let $k \geq 3$ be an odd integer and let  $G$ be a connected $(2t)$-regular graph, where $t \geq 1$. 
If $k \equiv 1 \pmod{4}$ then the graph $G$ is further required to be bipartite.
Then $o_v(\mathcal{M}_k(G)) = \frac{k + 1}{2}o_v(G)$ and $o_e(\mathcal{M}_k(G)) = o_e(G) + \frac{k + 1}{2}o_v(G)$.
Moreover, $|\Aut(\mathcal{M}_k(G))| = (2^t t!)^{|V(G)|} |\Aut(G)|$.
\end{proposition}

Proof of Proposition~\ref{prop:multiConstrSymGen} follows the same pattern as
the proof of Proposition~\ref{prop:multiConstrSym} and is left as an exercise to the reader.

Finally, Proposition~\ref{prop:multiConstrSymGen} implies some further results on the existence of infinite sets of graphs for given pairs $(o_v, o_e)$.
In Subsections~\ref{subsec:fam12} and~\ref{subsec:fam23} we provided infinite families of graphs
for which $(o_v, o_e) = (1, 2)$ and $(o_v, o_e) = (2, 3)$, respectively.
Using the machinery of multiplier constructions and their effects on symmetry, we obtain the next theorem.

\begin{theorem}
\label{prop:halfUniverseCharted}
Let $r \geq 2$ be even. For every $k \geq r + 1$ there exist infinitely many nut graphs $G$ with $o_v(G) = r$ and $o_e(G) = k$.
\end{theorem}

\begin{lemma}
\label{lemma:circulantDih}
For every $k \geq 1$ it holds that $\Aut(\Circ(n, \{1, 2, \ldots, k\})) \cong \Dih(n)$ for all $n \geq 2k + 3$.
\end{lemma}

\begin{proof}
Let $G = \Circ(n, \{1, 2, \ldots, k\})$. Recall that $V(G) = \mathbb{Z}_n = \{0, 1, 2, \ldots, n-1\}$.
It is clear that $\Dih(n) \leq \Aut(G)$.

Let $\mathcal{G} \leq \Aut(G)$. The Orbit-Stabiliser Lemma~\cite[Lemma~2.2.2]{Godsil2001} says $|\mathcal{G}| = |\mathcal{G}_v| \cdot |v^\mathcal{G}|$ for
any $v \in V(G)$. Since $G$ is vertex transitive it follows that $|\Aut(G)| = n \cdot |\Aut(G)_0|$. What is the orbit of vertex $1$ inside $\Aut(G)_0$?
Note that $d_G(0, i) = 1$ for $i \in \{1, 2, \ldots, k\} \cup \{-1, -2, \ldots, -k\}$ and $d_G(0, i) = 2$ for $i \in \{k+1, k+2, -k-1, -k-2\}$,
where $d_G(u, v)$ is the distance between vertices $u$ and $v$ in graph $G$. Let us define $f_0(i) = |\{ j \in N_G(i) \mid d(0, j) = 2 \}|$ for $v \in \{1, \ldots, k\} \cup \{ -1, \ldots, -k\}$.
Note that $f_0(1)  = f_0(-1) = 1$ and $f_0(i) \geq 2$ if $i \notin \{-1, 1\}$. Since graph automorphisms preserve distances, it follows that $1^{\Aut(G)_0} = \{-1, 1\}$, as these are the only two vertices at distance $1$
from $0$ that have a single neighbour at distance $2$. Therefore, $|\Aut(G)_0| = 2 \cdot |\Aut(G)_{0,1}|$, where $\Aut(G)_{0,1}$ is the stabiliser that fixes both $0$ and $1$.
It only remains to show that $\Aut(G)_{0,1}$  is trivial. Vertices $\{2, 3, \ldots, k + 1\} \cup \{ 0, -1, -k+1\}$ are at distance $1$ from vertex $1$. For these vertices we define
$f_1(i) = |\{ j \in N_G(i) \mid d(1, j) = 2 \}|$. The only vertex $\ell$ for which $d_G(\ell, 1) = 1$ and $f_1(\ell) = 1$ and $d_G(\ell, 0) = 1$ is the vertex $\ell = 2$. Therefore, 
 $\Aut(G)_{0,1}$ fixes vertex $2$. By iteration of the argument, all vertices are fixed, so $|\Aut(G)_{0,1}| = 1$.
\end{proof}

We remark in passing that $\Circ(2k+1, \{1, 2, \ldots, k\}) \cong K_{2k+1}$ and its automorphism group has order $(2k+1)!$; and that 
$\Circ(2k+2, \{1, 2, \ldots, k\}) \cong K_{2k+2} - (k+1)K_2$ and its automorphism group has order $2^{k+1}(k+1)!$.

\begin{proof}[Proof of Theorem~\ref{prop:halfUniverseCharted}]
For every $k \geq 1$ there exist infinitely many vertex transitive graphs with precisely $k$ edge orbits; 
they include the circulants $\Circ(n, \{1, 2, \ldots, k\})$ for $n \geq 2k+3$, provided by Lemma~\ref{lemma:circulantDih}.
By Proposition~\ref{prop:multiConstrSymGen}, 
\[
o_v(\mathcal{M}_{4q-1}(\Circ(n, \{1, 2, \ldots, k\}))) = 2q \text{ and } o_e(\mathcal{M}_{4q-1}(\Circ(n, \{1, 2, \ldots, k\}))) = k + 2q. \qedhere
\]
\end{proof}

\section{Future work}

The present paper gives a theorem for the relationship between vertex-orbit and edge-orbit counts for nut graphs.
The result (Theorem~\ref{thm:norb}) that
$o_e \geq o_v + 1$, compares to Buset's result $o_e \geq o_v - 1$ for all connected graphs \cite[Theorem 2]{Buset1985}.
Edge-transitive nut graphs are therefore impossible objects.

We also provided a complete characterisation of the orders for which nut graphs with $(o_v, o_e) = (1, 2)$ exist.
A partial answer was also found for the pair $(o_v, o_e) = (2, 3)$ (see Conjecture~\ref{conj:conjecturev2e3} and Question~\ref{quest:quest24}).
It was possible to provide infinite families of nut graphs for the pairs $(o_v, o_e)$, where $o_v$ is an even number and $o_e > e_v$. The case where $o_v$ is an odd number
remains to be completed. The ultimate goal is, of course, the complete characterisation of orders for all $(o_v, o_e)$ pairs.

During this work we encountered 
smallest examples of several interesting classes, including the non-Cayley nut graphs (see Figure~\ref{fig:nonCayleyExamples4}),
and GRR nut graphs (see Figure~\ref{fig:zero_sym_nut}), which suggest directions for future explorations.
The three infinite families used to prove Theorem~\ref{thm:fam1_2}, and the Rose Window family (see Proposition~\ref{prop:roseWindow}) are all quartic, but the problem of characterising orders for $(o_v, o_e)$ pairs is also a natural one for regular graphs,
or graphs of prescribed degree. It is planned to investigate the cubic case first, because of its significance for chemical graph theory.

A substantial part of the paper was devoted to constructions of nut graphs and their effects on
symmetry, which can be complicated. In some cases, the automorphism group of a constructed nut graph
can be impressively large (see Example~\ref{examp:2}). 
The multiplier constructions (Subsection~\ref{sec:multiConstructions})
give access to highly symmetric graphs with controlled number of vertex orbits. 
This prompts the question: For a given $n$, what is the most
symmetric nut graph on that order, where by `most symmetric' we mean in the sense of order of the automorphism group?
From this perspective it is interesting that the graph with $288$ automorphisms shown in Figure~\ref{fig:orbit_constr1}(i) is the 
nut graph with the largest
full automorphism group amongst all nut graphs on $10$ vertices and yet it is not vertex transitive.

\section*{Acknowledgements}
The work of Nino Bašić is supported in part by the Slovenian Research Agency (research program P1-0294
and research projects N1-0140 and J1-2481).
PWF thanks the Leverhulme Trust for an
Emeritus Fellowship on the theme of
{\lq Modelling molecular currents, conduction and aromaticity\rq}.
The work of Toma{\v z} Pisanski is supported in part by the Slovenian Research Agency (research program
P1-0294 and research projects N1-0140 and J1-2481).

\bibliographystyle{amcjoucc}
\bibliography{references}

\end{document}